\documentclass[11pt,a4paper]{amsart}
\newtheorem{thm}{Theorem}[section]
\newtheorem{prop}[thm]{Proposition}
\newtheorem{lem}[thm]{Lemma}
\newtheorem{cor}[thm]{Corollary}
\theoremstyle{definition}
\newtheorem{dfn}[thm]{Definition}
\theoremstyle{remark}
\newtheorem{rmk}[thm]{Remark}
\newtheorem{fact}[thm]{Fact}
\newtheorem{examp}[thm]{Example}
 \usepackage{amsfonts}
\usepackage{calligra}
\usepackage{amssymb}
\usepackage{verbatim}
\usepackage[colorlinks]{hyperref}
\usepackage[T1]{fontenc}
\usepackage{mathrsfs}
\usepackage{color}
\usepackage{mathrsfs}
\newcommand{\tr}{\,^t\!}

\newcommand{\g}{\mathfrak{g}}
\newcommand{\h}{\mathfrak{h}}
\renewcommand{\k}{{\mathfrak k}}
\renewcommand{\l}{\mathfrak l}
\newcommand{\p}{\mathfrak p}
\newcommand{\z}{\mathfrak z}
\renewcommand{\t}{{\mathfrak t}}

\newcommand{\C}{{\mathbb C}}
\newcommand{\N}{{\mathbb N}}

\newcommand{\R}{{\mathbb R}}
\newcommand{\Z}{{\mathbb Z}}

\numberwithin{equation}{section}

\begin{document}
\title{ Branching problems  in reproducing kernel spaces }
\author{  Bent {\O}rsted,   Jorge A.  Vargas}
\thanks{Partially supported by Aarhus University  (Denmark),  CONICET
(Argentina)  }
\date{\today }
\address{ Mathematics Department, Aarhus University, Denmark; FAMAF-CIEM, Ciudad Universitaria, 5000 C\'ordoba, Argentine}
\email{orsted@imf.au.dk,   vargas@famaf.unc.edu.ar}

\begin{abstract}
 For a semisimple Lie group $G$ satisfying the
equal rank condition, the most basic family of unitary
irreducible representations is the discrete series found
by Harish-Chandra. In this paper, we study some of the
branching laws for discrete series when restricted to a subgroup
$H$ of the same type by combining classical results
with recent work of T. Kobayashi; in particular, we prove
discrete decomposability under Harish-Chandra's
condition of cusp form on the reproducing kernel. We show a relation between discrete decomposability and representing certain intertwining operators in terms of differential operators.
\end{abstract}
\maketitle
\markboth{{\O}rsted- Vargas}{Branching problems and rep. kernels}

\tableofcontents

\section{Introduction}\label{sec:introduccion}

   Restricting a unitary irreducible representation
$\pi$ of a Lie group $G$ to a closed subgroup $H$ leads to
the {\it branching law problem}, namely  finding the explicit
decomposition of $\pi$ into irreducible representations
of $H$. This objective generalizes the theory of spectral decomposition
of a selfadjoint operator and, in the same way,   both a discrete and a continuous spectrum  may
occur.

In this paper, we shall consider the case of $G,$ a semisimple
Lie group, with $\pi$, a representation in the discrete series,
that is, $\pi$ occurs as a left-invariant closed irreducible subspace of $L^2(G)$;
these form the celebrated discrete series of Harish-Chandra.
$H$ will be a subgroup of the same type, and we shall
assume that both $G$ and $H$ admit discrete series.  An
aim is to find criteria ensuring that the branching law for $\pi$
gives irreducible $H$-invariant subspaces,  perhaps even that the
restriction of $\pi$ is a direct sum of such irreducible $H$-factors, in which case
we call this {\it discretely decomposable}, and, if this
happens with finite multiplicities, we call it {\it $H-$admissible}.
We shall combine the results of Harish-Chandra on the
distribution character and also the Plancherel formula
for $G$ with recent results of T. Kobayashi on the
admissibility of representations  to obtain some
new results on branching laws in our setting. Some of our proofs rely on work of N. Wallach published in \cite{Wa2}. We shall
also apply the theory of reproducing kernels, using some
specific models of the discrete series. This approach involves the
spherical functions  studied by Harish-Chandra. Note
that a special case of our situation could be with $H$ as
the diagonal subgroup in $G \times G$, so that the
branching problem corresponds to decomposing a tensor
product; this is already a complicated problem for
discrete series,  we hope our approach  leads to a
deeper understanding (as an easy, yet not completely
trivial consequence, we see that $\pi \otimes \pi^*$
is not discretely decomposable).

There are (at least) two basic facts from Harish-Chandra that
we use (1) convolution by the distribution character of $\pi$
gives the orthogonal projection onto the isotypic component of  $\pi$, and (2) the trace
of the spherical function for the lowest $K$-type $W$ is the
convolution (over $K$) of the distribution character and the
character of $W$. We also draw on  the groundwork on $H-$admissibility of T. Kobayashi \cite{K1}, \cite{K2}, \cite{K3}.   Our interest is only within the category
of unitary representations and spectral decomposition in
Hilbert spaces; thus, we do not consider the smooth category
and the corresponding {\it symmetry-breaking operators}
of T. Kobayashi. For this paper, a {\it symmetry breaking operator} is a continuous $H-$map from a unitary representation of $G$ into one of $H$. Concerning  either the structure of symmetry breaking operators or   the structure of the orthogonal projector  onto an isotypic component, we would like to mention the seminal work of Kobayashi-Pevzner \cite{KobPev}	
  \cite{KobPev2} \cite{KobPev3}, Kobayashi-Speh \cite{KobSp1} \cite{KobSp2},
Nakahama  \cite{N},   Peng-Zhang \cite{PZ}  and references therein.
  The    branching problem for  the tensor product of
two holomorphic discrete series of $G=SL_2(\R)$ has been studied by Molch$\check{\textrm{a}}$nov \cite{M2},  J. Repka and  other authors. An account of  techniques and results for this example  is in Kobayashi-Pevzner
\cite{KobPev3}.  In particular,  it is shown  that the expression of
the symmetry breaking operators employing Rankin-Cohen brackets
is helpful to make explicit  the corresponding   Plancherel formula.  In
  \cite{KobSp1} \cite{KobSp2}, Kobayashi-Speh analyze branching problems from $O(n,1)$ to $O(n-1,1)$ for general unitary irreducible representation $\pi$. In order to
study the  representations in the discrete spectrum they find an interesting
relation with the symmetry-breaking differential operators found by A. Juhl \cite{AJ}.
We may say that all of these examples, in addition to the
results on wavefront in \cite{HHO} \cite{HHO2},  have been a source of
inspiration for the questions this paper
considers. We want to point out that this paper is a sequel to \cite{OV} and that C.R.A.S., \cite{OV2}, has published a partial announcement of this paper.

We recall that applications of the knowledge of symmetry breaking operators to Harmonic Analysis have been  part of the research of several authors, among  them we mention \cite{KobPev3},  \cite{DaOZ}, the work of many authors on decomposing as direct integral of irreducible representations the left regular representation of $G$ on $L^2(G/H)$,   the papers of Molch$\check{\textrm{a} }$nov \cite{M1}, \cite{M2}.

   It remains an interesting question to understand the geometric nature of the differential
operators found to exist in connection with symmetry breaking for the family of
representations considered here, i.e., they live in spaces of sections over Riemannian
symmetric spaces. In addition to finding explicitly these operators, it might be
possible to extend their existence to more general Riemannian manifolds and
submanifolds - this should then include curvature of such more general manifolds.

We now mention some of the main results in this paper.

This paper splits into twelve sections. The results in Sections 2.0, 5.0, and  subsections 3.2, 6.2  are quite technical. In subsection 3.1, we show that under certain hypotheses, symmetry breaking operators agree on  dense subspaces with   integral operators. In all the  sections, we use the fact that a discrete series representation admits a model on a space of solutions to  elliptic differential operators. For example,   we carry out an analysis of symmetry breaking operators as differential operators (subsection 4.1, 4.2, 4.3). In particular, we show that $H-$admissibility implies that all the elements of some families of  symmetry breaking operators are represented by differential operators. We also show a converse statement.   It is quite  a different  situation when we analyze the orthogonal projector onto an isotypic component (Examples 6.2, 6.3); under the hypothesis  of admissibility, we may show projectors are represented via truncated Taylor series (section 6) or some version of infinite order differential operators. Concerning   necessary and sufficient conditions for a representation to be $H-$discretely decomposable, we obtain criteria saying that some particular symmetry breaking operators are restrictions of differential operators (Theorem 4.3) or in terms of reproducing kernels and cusp forms of Harish-Chandra (Theorem 7.1).  We  observe  that in the family of discrete series representations, the hypothesis $H$-admissible implies $H-$smooth vectors in a isotypic component are $G-$smooth vectors. In section 8, we introduce two functions that are suitable for checking   the existence of discrete factors (Theorem 8.4). In two appendices, we collect results on integral operators, reproducing kernel subspaces  and kernel of elliptic invariant differential operators, as well as notation. The main results in this paper are: Theorems 3.5, 4.3, Proposition 4.4, Theorem 4.11,  Propositions 5.2, 6.6, 6.7, Theorem 7.1, Corollary 7.4.

\section{Preliminaries }\label{sec:prelim}

Let $G$ be an arbitrary, matrix,  connected semisimple Lie group. Henceforth we fix a maximal compact subgroup $K$ for $G$  and a maximal torus $T$ for $K.$   Harish-Chandra showed that $G$ admits square integrable irreducible representations if and only if $T$ is a Cartan subgroup of $G.$  For this paper,  we always assume $T$ is a Cartan subgroup of $G.$ Under these hypotheses, Harish-Chandra  showed that  the set of equivalence classes of irreducible square integrable representations is parameterized by a lattice in $i\mathfrak t_\mathbb R^\star.$ To state our results, we need to make  explicit this parametrization and set up some notation.   As usual,  the Lie algebra of a Lie group is denoted by the corresponding lower case German letter followed by the subindex $\mathbb R.$  The complexification of  the Lie algebra of a Lie group is  denoted by the corresponding German letter without any subscript.     $V
^\star $ denotes the dual space to a vector space $V.$ Let $\theta$ be the Cartan involution which corresponds to the subgroup $K,$ the associated Cartan decomposition is denoted by $\g=\k +\p.$ Let $\Phi(\g,\t) $ denote the root system attached to the Cartan subalgebra $\t.$ Hence, $\Phi(\g,\t)=\Phi_c \cup \Phi_n =\Phi_c(\g, \t) \cup \Phi_n (\g, \t)$ splits up as the union the set of compact roots and the set of noncompact roots. From now on, we fix a system of positive roots $\Delta $ for $\Phi_c.$   For this paper, either the highest weight or the infinitesimal character of an irreducible representation of $K$ is  dominant with respect to $\Delta.$ The Killing form gives rise to an inner product $(...,...)$ in $i\t_\mathbb R^\star.$ As usual, let $\rho=\rho_G $ denote half of the sum of the roots for some system of positive roots for $\Phi(\g, \t).$  \textit{A Harish-Chandra parameter} for $G$ is $\lambda \in i\t_\mathbb R^\star$ such that $(\lambda, \alpha)\not= 0 , $ for every $\alpha \in \Phi(\g,\t) ,$    and so that $\lambda + \rho$ lifts to a character of $T.$ To each Harish-Chandra parameter $\lambda$, Harish-Chandra, associates a unique irreducible square integrable representation $(\pi_\lambda^G , V_\lambda^G)$ of $G$ of infinitesimal character $\lambda.$ Moreover, he showed the map $\lambda \rightarrow (\pi_\lambda^G, V_\lambda^G)$  is a bijection from the set of Harish-Chandra parameters dominant with respect to $\Delta$  onto the set of equivalence classes of irreducible square integrable representations for $G$ (cf. \cite[Chap 6]{Wa1}). For short, we will refer to an irreducible square integrable representation as a discrete series representation.

 Each Harish-Chandra parameter $\lambda$ gives rise to a system of positive roots \\
\phantom{xxxxxxxxxxxxxxx}$\Psi_\lambda =\Psi_{G,\lambda} =\{  \alpha \in \Phi(\mathfrak g, \mathfrak t) : (\lambda, \alpha) >0 \}.$ \\ From now on, we assume that Harish-Chandra parameter for $G$ are dominant with respect to  $\Delta.$ Whence,  $\Delta \subset \Psi_\lambda.$

 Henceforth,    $\Psi$ is a system of positive roots for $\Phi (\g, \t)$ containing $\Delta,$ and
 $(\pi_\lambda^G , V_\lambda^G )$ a square integrable representation for $G$ of Harish-Chandra parameter $\lambda$ dominant with respect to $\Psi.$  $ (\tau ,W ):= (\pi_{\lambda +\rho_n}^K , V_{\lambda + \rho_n}^K) $ denotes the lowest $K-$type of $\pi_\lambda :=\pi_\lambda^G.$ The highest weight of $(\pi_{\lambda +\rho_n}^K , V_{\lambda + \rho_n}^K)$ is $\lambda +\rho_n -\rho_c.$  We recall a Theorem of Vogan's thesis which states that $(\tau,W)$ determines $(\pi_\lambda, V_\lambda^G)$ up  to unitary equivalence \cite{Vo}.   We recall the set of square integrable sections of the vector bundle determined by the principal bundle $K\rightarrow G \rightarrow G/K$ and the representation $(\tau, W)$ of $K$ is isomorphic to the space \begin{multline*}L^2(G\times_\tau W):= \{ f \in L^2(G) \otimes W : \\ f(gk)=\tau(k)^{-1} f(g),   g  \in G, k \in K \}.
   \end{multline*}
   Here, the action of $G$ is by left translation $L_x, x \in G.$   The inner product on $L^2(G)\otimes W$ is given by \begin{equation*}(f,g)_{V_\lambda} =\int_G (f(x),g(x))_W dx, \end{equation*} where $(...,...)_W$ is a $K-$invariant inner product on $W.$
   Subsequently, $L_D $ (resp. $R_D)$ denotes the left infinitesimal (resp. right infinitesimal) action on functions from $G$ of an element  $D$ in universal enveloping algebra $U(\g)$ for the Lie algebra $\g$.  As usual, $\Omega_G$ denotes the Casimir operator for $\g.$   Following Hotta, Enright-Wallach \cite{OV}, we realize $V_\lambda :=V_\lambda^G $ as the space
\begin{multline*}
  H^2(G, \tau) =\{ f \in L^2(G) \otimes W : f(gk)=\tau(k)^{-1} f(g) \\ g\in G, k \in K, R_{\Omega_G} f= [(\lambda, \lambda) -(\rho, \rho)] f  \}.
 \end{multline*}
Henceforth, for $f \in H^2(G,\tau)$, we write $\pi_\lambda (x)(f):=L_x (f).$ We  recall $R_{\Omega_G} = L_{\Omega_G} $ is an elliptic $G-$invariant operator on the vector bundle $W \rightarrow G \times_\tau W \rightarrow G/K$ and hence,  $\,H^2(G,\tau)$ consists of smooth sections,   moreover point evaluation $e_x$ defined by  $ \,H^2(G,\tau) \ni f \mapsto f(x) \in W $ is continuous for each $x \in G$ (cf. Appendix~\ref{subsec:A4}). Therefore, the orthogonal projector $P_\lambda$ onto $\,H^2(G,\tau)$ is an integral map (integral operator) represented by the smooth Carleman {\it matrix  kernel} or {\it reproducing kernel} (cf. Appendix~\ref{subsec:A1}, Appendix~\ref{subsec:A4}, Appendix~\ref{subsec:A6}). \begin{equation} \label{eq:Klambda}K_\lambda : G\times G \rightarrow End_\C (W) \end{equation} which satisfies $  K_\lambda (\cdot ,x)^\star w$ belongs to $\,H^2(G,\tau)$ for each $x \in G, w \in W$ and $$ (P_\lambda (f)(x), w)_W=\int_G (f(y), K_\lambda (y,x)^\star w)_W dy, \,   f\in L_2(G\times_\tau W).$$
For a closed reductive subgroup  $H$, after conjugation by an inner automorphism of $G$ we may and will assume  $  L:=K\cap H $ is a maximal compact subgroup for $H.$ That is, $H$ is $\theta-$stable. In this paper for irreducible square integrable representations $(\pi_\lambda, V_\lambda)$  for $G $ we  analyze the restriction to $H.$ In particular, we study the irreducible $H-$subrepresentations for $\pi_\lambda$. A known result is that any irreducible $H-$subrepresentation of $V_\lambda$ is a square integrable representation for $H$, for a proof  (cf. \cite{GW}). Thus, owing to the result of Harish-Chandra on the existence of square integrable representations,  from now on, we  may and will assume {\it $H$ admits a compact Cartan subgroup}. After conjugation, we may assume $U:=H\cap T$ is a maximal torus in $L=H\cap K.$ Next,     we consider a square integrable representation $H^2(H, \sigma) \subset L^2 (H \times_\sigma Z)$ of lowest $L-$type $(\sigma, Z).$  An aim of this paper is to understand the nature of the intertwining operators between the unitary $H-$representations $H^2(H,\sigma)$ and $\,H^2(G,\tau), $ the adjoint of such intertwining operators and consequences of their structure.

\begin{rmk} To give a glance on the nature of the elements in $\,H^2(G,\tau)$, we show: Every $f \in \,H^2(G,\tau)$ is a bounded function.\\
	In fact, let $K_\lambda (y,x)$ be as in  \ref{eq:Klambda}.  Thus, $ f(x)=\int_G K_\lambda(y,x) f(y) dy.$ Then, Schwarz inequality,  $K_\lambda (\cdot,x)$ is square integrable, the equality  $K_\lambda (y,x)=K_\lambda (x^{-1}y,e)$  and the  invariance of Haar measure, justify  the inequalities
	\begin{alignat*}{2}
	\Vert f(x) \Vert_W  & \leq  \int_G \Vert  K_\lambda (y,x)\Vert_{Hom(W,W} \Vert f(y) \Vert_W  dg \\
	& \, \leq (\int_G \Vert K_\lambda (y,x) \Vert^2 dy )^\frac12  \, (\int_G \Vert f(y) \Vert^2 dy)^\frac12 \\ & \, \leq (\int_G \Vert K_\lambda (y,e) \Vert^2 dy )^\frac12  \, (\int_G \Vert f(y) \Vert^2 dy)^\frac12 \\ & \, = \Vert K_\lambda (\cdot,e)\Vert_{L^2(G)} \Vert f \Vert_2.
	\end{alignat*}
\end{rmk}

\section{Structure  of intertwining maps}
  For this section, $G,K, (\tau, W), (\pi_\lambda,\,H^2(G,\tau))=(\pi_\lambda^G ,V_\lambda^G), H, L$ are as in  Section~\ref{sec:prelim}. Besides, we fix   a finite dimensional representation $(\nu, E)$ for $L$, and  a continuous intertwining linear $H-$map  $T : L^2(H\times_\nu E)\rightarrow \,H^2(G,\tau)$. \begin{fact}    We show $T$ is a Carleman kernel map. (cf. Appendix~\ref{subsec:A1}). \end{fact}
In fact,  for each $x \in G, w \in W$, the linear functional  $L^2(H\times_\nu E) \ni g \mapsto (Tg(x),w)_W$ is continuous.  Riesz representation Theorem shows there exists a function  $$K_T : H\times G \rightarrow Hom_\C (E,W)$$ so that
the map $ h\mapsto K_T(h,x)^\star(w)$ belongs  to $L^2(H \times_\nu E)$ and for $g \in L^2(H\times_\nu E), x \in G,  w \in W$  we have the absolutely  convergent integral and the equality
\begin{equation}\label{eq:tisintegral}
(Tg(x),w)_W    =\int_H (g(h), K_T(h,x)^\star w)_Z  dh.
\end{equation}
That is, $T$  is the integral map  $$Tg(x)=\int_H K_T(h,x) g(h) dh,  x \in G.$$

\subsection{Symmetry breaking operators}

Examples shows that the adjoint $T^\star$ of   integral linear map $T$  need not be an integral map (cf. Appendix 10.3), whence, we would like to know when $T^\star$ is an integral linear map.     Formally, we may  write $T^\star f (h) =\int_G K_T(h,x)^\star f(x) dx,$ where the convergence of the integral is in the weak sense. That is, for each $g \in L^2(H\times_\nu E), f \in \,H^2(G,\tau),$ we have the  convergence of the iterated integral $$(T^\star f,g)_{L^2(H\times_\nu E)} = \int_G  \int_H (f(x), K_T(h,x)g(h))_W dh dx. $$   Thus, the adjoint $T^\star$ of any continuous $T : L^2(H\times_\nu E) \rightarrow \,H^2(G,\tau)$ is a weak integral map. \\ In order to study branching problems,   T. Kobayashi has introduced  the concept of symmetry breaking operator. In our setting, a {\it symmetry breaking operator} is a continuous $H-$map
$S : \,H^2(G,\tau) \rightarrow L^2(H\times_\nu E)$. For a symmetry breaking operator $S$, the above considerations applied to $T:=S^\star$  let us conclude: under our hypothesis,  a symmetry breaking operator is always a weak integral map.\\
In \cite[page 45]{Fo}, we find an exposition on   Bargmann transform $ B :L^2(\mathbb R^n) \rightarrow \mathcal H_2(\mathbb C^n)$. $B$ is an example of  integral map,  such that  $B^\star$ may not be an integral map. However, Bargmann has shown  $B^\star $ restricted to certain dense subspace is an integral map. In Appendix~\ref{subsec:A3}, we provide an example of  a continuous integral map into a reproducing kernel space such is  that its adjoint is not an integral map on the whole space. However, it is equal to an integral map on a certain subspace. Theorem~\ref{prop:symbrea} shows that a quite common feature for the kind of integral maps $T$ under our consideration, that its adjoint $T^\star$ is an integral map on certain dense subspace of $\,H^2(G,\tau).$ In order to state our results we need a few definitions.
\begin{dfn} \label{dfn:disdecom} A unitary representation  $(\pi , V)$  of $G$   is discretely decomposable over $H$, if there exists an orthogonal family of closed, $H-$invariant, $H-$irreducible subspaces of $V$ so that the closure of its algebraic sum is equal to $V.$   \end{dfn}
\begin{dfn}  A unitary  representation $(\pi, V)$ of $G$ is $H-$admissible if the representation is discretely decomposable over $H$ and the  multiplicity of each irreducible $H$-factor is finite. \end{dfn}
In \cite{KO}, we find a complete list of triples $(G,H, \pi)$ such that $(G,H)$ is a symmetric pair and $\pi$ is an $H-$admissible representation.   For example, for the pair $(SO(2n,1), SO(2k)\times SO(2n-2k,1))$ there is no $\pi_\lambda$ with an admissible restriction to $H.$  Whereas, for the pair $(SU(m,n), S(U(m,k) \times U (n-k))) $ there are exactly $m$  Weyl chambers $C_1,\dots,C_m$ in $i\mathfrak t_\R^\star$, so that $\pi_\lambda$ is $H-$admissible if and only $\lambda$ belongs to   $C_1\cup \dots \cup C_m.$ For both cases $0<k<n$.
\begin{dfn} \label{dfn:integ} A   representation $(\pi , V)$ of $G$ is integrable, if some nontrivial  matrix coefficient is an integrable function with respect to Haar measure. \end{dfn} It is a Theorem of Trombi-Varadarajan, Hecht-Schmid, \cite{HS}:  $( \pi_\lambda, V_\lambda)$ is an integrable representation if and only if   $\vert (\lambda, \beta)\vert > \frac12 \sum_{\alpha \in  \Phi(\g,\t): (\alpha, \beta)>0} (\alpha, \beta)$ for every noncompact root $\beta.$\\
 The next result gives more information on symmetry breaking operators. We  recall (cf. section 2) that the action $\pi_\lambda$ of $G$    on $H^2(G,\tau)$ by left translation  provides an explicit realization of $(\pi_\lambda, V_\lambda)$.
\begin{thm} \label{prop:symbrea} Let  $S : \,H^2(G,\tau) \rightarrow L^2(H\times_\nu E)$ be a continuous intertwining linear $H-$map. Then,
	
	a) If the restriction to $H$ of $(\pi_\lambda,  \,H^2(G,\tau))$ is discretely decomposable, then $S$ is an integral map.
	
	b) If $(\pi_\lambda, \,H^2(G,\tau))$ is an integrable representation, then $S$ restricted to the subspace of smooth vectors is an integral linear map.
\end{thm}

We defer  the proof of Theorem~\ref{prop:symbrea} to section 9.
 \begin{examp} An application for  Theorem~\ref{prop:symbrea} is: for a continuous $H-$map, we write the polar decomposition for $T=VP : L^2(H\times_\sigma Z)\rightarrow \,H^2(G,\tau)$ where $P=\sqrt{T^\star T}$ and $V :L^2(H \times_\sigma Z) \rightarrow \,H^2(G,\tau)$ is a partial isometry. $V$ is usually called a generalized Bargmann transform.  Then, $V$ is an integral map, and,  whenever $\pi_\lambda^G$ is an integrable discrete series, the linear map $V^\star$,  as well as $T^\star$, restricted to   the subspace of smooth vectors in $\,H^2(G,\tau)$, is equal to  an integral map. A particular case of this is $T=r^\star$,  where $r$ is the restriction map $r  :\,H^2(G,\tau) \rightarrow L^2(H\times_{res_L(\tau)} W)$ as in Example~\ref{examp:rn}. Here, $r^\star =V \sqrt{rr^\star}$, $r=V^\star \sqrt{r^\star r}$,  and $r, r^\star, rr^\star, r^\star r, \sqrt{r^\star r}, \sqrt{rr^\star}   $   are integral maps (cf. \cite{He}). Theorem~\ref{prop:symbrea} shows that   when $\pi_\lambda$ is integrable, $V^\star$  is a kernel map on the subspace of smooth vectors.  For $G=SU(1,1), H=A$, in Example~\ref{examp:sl2}, we verify  $V^\star$ restricted to the subspace of $K-$finite vectors is  an integral map, despite $\pi_\lambda $ is not an integrable representation. \end{examp}

\subsection{Operators from $L^2(H\times_\nu E)$ into $\,H^2(G,\tau)$} In this somewhat technical  subsection we assemble   properties of the Carleman matrix kernel that represent  a $H-$map into $\,H^2(G,\tau)$ and its adjoint.
\begin{prop} \label{prop:propertieskt}  Let   $T : L^2(H\times_\nu E)\rightarrow \,H^2(G,\tau)$ be a  continuous intertwining linear $H-$map. Then, the function $K_T$ satisfies:\\
	a) $K_T(h,x)^\star w =T^\star (y \mapsto K_\lambda (y,x)^\star w)(h).$\\
	b) The function $h  \mapsto K_T(h,e)^\star w$ is an $L-$finite vector in $L^2(H\times_\nu E).$\\
	c) $K_T$ is a smooth map. Further, $K_T(\cdot,x)^\star w$ is a smooth vector. \\
	d) There exists a constant $C$ and finitely many functions $\phi_{a,b} : G\rightarrow \C$ so that for every $x \in G$,
	$\Vert K_T(e,x)^\star \Vert_{Hom_\C (W,Z)}  \leq C \Vert T^\star \Vert   \sum_{a,b} \vert  \phi_{a,b}(x)\vert . $ \\
	e) The function  $ G \ni x \mapsto \Vert K_T(\cdot,x)^\star \Vert_{L^2(H \times_{\tau^\star \otimes\nu } Hom_\C (W,E))} $ is bounded.\\
	f) $K_T (hh_1s, hxk)=\tau(k^{-1})K_T(h_1,x)\nu(s), \, x\in G, s \in L, h,h_1 \in H, k\in K.$\\ g) If $T^\star$ is a kernel map, with kernel $K_{T^\star} : G\times H \rightarrow Hom(W,E)$ and $K_{T^\star}(\cdot,h)^\star z \in \,H^2(G,\tau).$ Then, $L_D^{(2)} K_T(h,\cdot) =\chi_\lambda (D) K_T(h,\cdot)$ for every $D$ in the center of $U(\g).$ Here, $\chi_\lambda $ is the infinitesimal character of $\pi_\lambda.$
\end{prop}
{\it   Note.} The functions $\phi_{a,b}$ are defined as follows. We fix  linear basis $\{X_b\}_{1\leq b \leq N}$ (resp. $\{Y_a\}_{1\leq a \leq M}$) for the space of elements in $U(\h)$ of degree less or equal than $\dim \h$ (resp.  for the space of elements in $U(\g)$ of degree less or equal than $\dim \h$). Then $\phi_{a,b}$ are defined by $Ad(x^{-1})(X_b) =\sum_{1 \leq a \leq M} \phi_{a,b}(x) Y_a, b=1,\cdots ,N.$

We postpone the verification of Proposition~\ref{prop:propertieskt} to section 9.

\begin{rmk}\label{rmk:structS} We assume $S :\,H^2(G,\tau) \rightarrow L^2(H\times_\nu E)$ is a continuous symmetry breaking operator represented by a Carleman kernel $K_S$. Thus, owing to Proposition~\ref{prop:propertieskt} $S^\star$ is a Carleman map. Hence,  $K_S(x,h)=K_{S^\star}(h,x)^\star$  and hence we may conclude:\\ a) $K_S$ is a smooth map.\\
	b) $K_S(hxk,hh_1s)=\nu(s^{-1})K_S(s,h_1)\tau (k), h,h_1 \in H, x\in G, s \in L.$\\
	c) For $z \in E$, the function $G \ni x \mapsto K_S(x,e)^\star z \in W $ belongs to $\,H^2(G,\tau)$ and it is a $L-$finite function.
\end{rmk}

\begin{rmk} Any continuous $H-$map from $\,H^2(G,\tau)$ into $H^2(H,\sigma)$ or vice versa,  always, it is  represented by a smooth Carleman kernel that enjoys the properties in Remark~\ref{rmk:structS} or Proposition~\ref{prop:propertieskt}. We also show, in Theorem~\ref{prop:diffop}, that symmetry breaking operators  can be represented by differential operators under the hypothesis of $res_H(\pi_\lambda)$ being $H-$admissible. \\ With respect to compute the dimension of the space of symmetry breaking operators, we refer to \cite{DV}, \cite{DGV}, \cite{GW}, \cite{N} and references therein. Quite often, the dimension of the space of symmetry breaking operators is large, (cf. \cite{DGV}) and references therein.
\end{rmk}
\subsection{Generalized Shintani functions}
Kobayashi in \cite{K5} has began  deep analysis  of functions in relation  to $H-$intertwining linear operators $R :V_G \rightarrow V_H$  between two smooth irreducible and class one representations  $(\pi_G, V_G)$ $(\pi_H, V_H)$ for $G,$ $H$ respectively.  For this, he fixes  $v_G \in V_G, v_H \in V_H$ nonzero vectors fixed by   $K,  L$ respectively, and he defines the Shintani function $S(g)=(R(\pi_G(g)v_G),v_H)_{V_H}, g \in G.$ The function $S$ satisfies:  right invariant under $L$; left invariant under $K$; eigenfunction for $R_D, D \in \z(U (\g))$;  eigenfunction for $L_D,  D \in \z(U (\h)).$ Here, $\z(U (\g))$ denotes the center of the universal enveloping algebra for $\g.$ Let $S: \,H^2(G,\tau) \rightarrow H^2(H,\sigma)$ a continuous intertwining linear map. Since $H^2(H,\sigma)$ is a reproducing kernel subspace, $S$ is an integral map represented by a kernel $K_S : G\times H \rightarrow Hom_\C(W,Z)$ that satisfies the statements in Remark~\ref{rmk:structS}.  We  now verify  that properties of the  function $y \mapsto K_S(y,e)$ suggest  a natural generalization for the concept of Shintani function. In fact, we set $\Phi(y):=K_S(y,e), y \in G.$ Thus, $\Phi$ maps $G$ into $Hom_\C (W,Z)$. Among the properties of the function $\Phi$ are:
\begin{enumerate}
\item $\Phi$ is a smooth function.
  \item  $\Phi(syk)= \sigma(s) \Phi(y) \tau(k), s \in L, k \in K, y \in G.$
  \item $R_D \Phi = \chi_\lambda (D) \Phi $ for $D \in  \z(U(\g)).$
  \item  $L_D \Phi = \tilde{\chi}_\mu (D) \Phi$ for $D \in  \z(U(\h)).$
  \item $\Phi $ as well as its restriction to $H$ are square integrable functions.
\end{enumerate}
The previous considerations let us  define a {\it generalized  Shintani function}. For this, we fix representations $(\tau,W)$ of $K$,  $(\sigma, Z)$ of $L$ and infinitesimal characters $\chi_\lambda, \chi_\mu$ for   $\g, \h$ respectively. {\it A generalized Shintani function} is  a function $\Phi : G\rightarrow Hom_\C(W,Z) $ that satisfies the five  conditions enumerated in the previous paragraph.
When we dealt with discrete series representations, the trivial representation of a maximal compact subgroup of the ambient group, never occurs as a $K-$type. However, it is clear that the first four stated properties   are a generalization of the concept of Shintani's function. The obvious result is:  {\it  the space of generalized Shintani functions attached to $(\tau, W), (\sigma ,Z), \chi_\lambda, \chi_\mu$ is isomorphic to the space of continuous $H-$maps from $V_\lambda^G $ to $V_\mu^H.$}
 \section{Intertwining operators via differential operators} For this section,   let $G,K,H,L, W, \pi_\lambda,   \,H^2(G,\tau) $ be  as in Section~\ref{sec:prelim}. Besides, we fix   a finite dimensional representation $(\nu, E)$ for $L$. In \cite{N}, \cite{KobPev}, \\ \cite{K6},  the authors have constructed $H-$intertwining maps between holomorphic discrete series by means of differential operators. Some of these authors also considered the case of intertwining maps between two principal series representations \cite{KobSp1}, \cite{KobSp2}. Motivated by the fact that discrete series can be modeled as function spaces,  an  aim of this section is to analyze to what extent symmetry breaking linear maps agree with restriction of linear differential operators. In \cite{K6},  is presented a general conjecture on the subject, we present a partial solution to the conjecture for the particular case of discrete series representations.
\subsubsection{  Differential Operators.}
For the purpose of this paper {\it a  differential operator} is a linear map  $S : C^\infty (G \times_\tau W ) \rightarrow C^\infty (H\times_\nu E)$ so that there exists   finitely many elements  $D_b \in U(\g),$   a basis $ \{w_c \} \, \text{for}\, W,$   finitely \, many \, el\-e\-ments $ \{z_a \} \, \text{in}\,  E,$ $d_{a,b,c} \in \C,$ and for any  $f\in C^\infty (G\times_\tau W)$ we have \begin{equation} \label{eq:defdop} S(f)(h)= \sum_{a,b,c} d_{a,b,c}\, ([R_{D_b}f](h),w_c)_W\, z_a \,\, \forall h \in H. \end{equation}
  Sometimes we will allow the constants $d_{a,b,c}$ to be smooth functions on $H.$ \\  The definition is motivated by the following result of Kobayashi, for a proof (cf.  \cite[Theorem 2.9]{KobPev}).  They present  an isomorphism between  the space $Hom_L (E^\vee , Ind_\k^\g(W^\vee))$ and  the set of $H-$invariant differential operators from $C^\infty (G\times_\tau W) $ into $C^\infty (H\times_\nu E)$. As usual, $M^{\vee}$ denotes the contragredient representation. The isomorphism is: \begin{multline*} \sum_{a,b,c} d_{a,b,c}\, z_a \otimes D_b \otimes w_c \\ \mapsto (C^\infty (G\times_\tau W) \ni f \mapsto (h \mapsto \sum_{a,b,c} d_{a,b,c}\, ([R_{D_b}f](h),w_c)_W\, z_a)). \end{multline*} Here, $d_{a,b,c} \in \C, \{z_a\} \, \text{is basis for}\, E, D_b \in U(\g),  \{w_c\} \, \text{is basis  for} \, W, h \in H   .$\\
    It is obvious that a differential operator according to our definition,  with $d_{a,b,c}  \in \C \,\, \forall a,b,c$  is represented  by an  element of $Hom_L (E^\vee , Ind_\k^\g(W^\vee))$.
 \begin{examp} \label{examp:rn} Examples of differential operators are the normal derivatives  considered in \cite{OV}. For this,
	we write the Cartan decomposition as  $\g=\k +\p$ and $\h =\l +\p^\prime$. We have $\p^\prime =\p \cap \h.$  Let $(\p/\p^\prime)^{(n)} $ denote the $n-$th symmetric power of the orthogonal  to $\p^\prime$ with respect to the Killing form in $\p.$ Let $\tau_n$ denote the natural representation of $L$ in $Hom_\C ( (\p/\p')^{(n)} , W) $. Let $\lambda : S(\g) \rightarrow U(\g)$ denote the symmetrization map. Then,  for each $ D \in (\p/\p')^{(n)},$ $f \in \,H^2(G,\tau), h \in H$ we compute the normal derivative of $f$ in the direction $D$ at the point $h,$  $r_n(f)(D)(h) :=R_{\lambda (D)} f(h) .$ In \cite{OV}, it is shown that $r_n(f) \in L^2(H\times_{ \tau_n}  Hom_\C ( (\p/\p^\prime)^{(n)}, W))$  and   the resulting map $$r_n : \,H^2(G,\tau) \rightarrow L^2(H \times_{\tau_n}  Hom_\C ( (\p/\p^\prime)^{(n)}, W))$$ is  $H-$equivariant and continuous for $L^2-$topologies. As before, $K_\lambda$ is the matrix kernel of  $P_\lambda$.
	The map $r_n$ is represented by the matrix kernel $$K_{r_n} : G\times H \rightarrow Hom_\C ( W, Hom_\C ((\p/\p^\prime)^{(n)}, W)) $$ given by $$ K_{r_n}(y,h)(w,D) =R_{\lambda (D)}^{(2)}(K_\lambda (y,h)w).  $$ Here, the upper index $(2)$ means we compute the derivative of $K_\lambda$ on the second variable.
\end{examp}

\subsection{Symmetry breaking via differential operators}

In this subsection we show the equivalence between $res_H(\pi_\lambda)$ being a $H-$admissible representation and that every symmetry breaking operator from $\,H^2(G,\tau)$ into an arbitrary representation  $H^2(H,\sigma)$  is the restriction of a differential operator. For this, (cf. Lemma~\ref{prop:diffopiskern}), we characterize those symmetry braking operators that are restriction of differential operators. Next, we formulate the main result in Theorem~\ref{prop:diffop}.   We recall (cf.  Theorem~\ref{prop:symbrea} and \ref{rmk:structS} c))  we have shown  that if
$res_H(\pi_\lambda)$ is a $H-$discretely decomposable representation, then any symmetry breaking operator $S$ is a Carleman integral map satisfying:    for every $z\in E$, $y\mapsto K_S(y,e)^\star z$ is a $L-$finite vector for $\pi_\lambda.$
  \begin{lem}\label{prop:diffopiskern} $G,H, \,H^2(G,\tau), L$ are  as in Section~\ref{sec:prelim}. $(\nu,E)$ is a finite dimensional representation for $L$. Let $S : \,H^2(G,\tau) \rightarrow L^2(H\times_\nu E)$ be  a  not necessarily
  continuous intertwining  $H-$map   such that  $S$ is
	the restriction of a differential operator. Then,\\
	a) $S$ is a Carleman map. That is,   there exists  $K_S :G\times H \rightarrow Hom_\C (W,E)$ so that $y \mapsto
K_S(y,h)^\star z \in \,H^2(G,\tau) \,\text{for}\, h \in H, z \in E$ and for $\, f \in \,H^2(G,\tau)$ \\
	  $(S(f)(h),z)_Z =\int_G (f(y), K_S(y,h)^\star z)_W dy \,  z \in E.$\\
	b) For every $z\in E$, $y\mapsto K_S(y,e)^\star z$ is a $K-$finite vector for $\pi_\lambda.$\\
	c) $K_S$ is a smooth function.\\
	d) $S$ is continuous in $L^2-$topologies. \\
	Conversely, if $S : \,H^2(G,\tau) \rightarrow L^2(H\times_\nu E)$ is a Carleman $H-$map so that  for each $z \in E$, $y\mapsto
K_S(y,e)^\star z$ is a $K-$finite vector for $\pi_\lambda.$ Then,    $S$ is the restriction
of a differential operator and  $S$ is continuous in $L^2-$topologies.
\end{lem}
    We defer the proof of Lemma~\ref{prop:diffopiskern} till we
formulate the main  results for this subsection.

\begin{thm}\label{prop:diffop} We recall $G$ is a semisimple matrix Lie group, $(\pi_\lambda, \,H^2(G,\tau))$ is an irreducible square integrable representation for $G$,  $ H$ is a closed, equal rank,  reductive subgroup of $G$ and $L$ is a maximal compact subgroup for $H$.  We have:   $res_H(\pi_\lambda)$ is an $H-$admissible representation if and only if for every irreducible representation $(\sigma, Z) $ of $L$, it holds that any continuous $H-$map from  $ \,H^2(G,\tau)$ into $H^2(H,\sigma)$ is the restriction of a differential operator.
\end{thm}	

Theorem~\ref{prop:diffop} readily follows from  Lemma~\ref{prop:diffopiskern} and the next Proposition.

In \cite[Theorem 3.10]{N},  Nakahama   has shown the direct implication in  Theorem~\ref{prop:diffop} under the hypotheses  $(G,H)$ is a holomorphic symmetric pair and $\pi_\lambda$ is a scalar holomorphic discrete series.

 The proof of Theorem~\ref{prop:diffop} is based in the following fact.
\begin{prop}\label{prop:adviadiffop}	Let $G,K, \,H^2(G,\tau), \pi_\lambda, H,L,   (\nu ,E)$ be as is in Lemma~\ref{prop:diffopiskern}. Let $S : \,H^2(G,\tau) \rightarrow L^2(H\times_\nu E)$ denote an everywhere defined intertwining linear
$H-$map. The following statements holds: \\     If  $res_H(\pi_\lambda)$ is an $H-$admissible representation and
$S$ is continuous. Then  $S$ is  the restriction of a linear differential operator. \\  For a converse statement, we have,\\
	i) If we assume for some $\nu$, that,  some nonzero linear  intertwining $H-$map $S :\,H^2(G,\tau) \rightarrow
L^2(H\times_\nu E)_{disc}$ is the restriction of a linear differential operator. Then,  $res_H(\pi_\lambda)$ is
discretely decomposable.\\
	ii) If  we assume for some irreducible representation $(\sigma ,Z)$ for $L$,    that, every nonzero linear intertwining $H-$map $S :\,H^2(G,\tau)
\rightarrow H^2(H,\sigma)$ is the restriction of a linear differential operator. Then,  the multiplicity of $
(\pi_\mu^H ,H^2(H,\sigma))$ in $res_H (\pi_\lambda)$ is finite.
\end{prop}

  In Proposition~\ref{prop:adviadiffop}, the hypothesis on image of $S$ is  quite essential. Examples and counterexamples are provided by
$r, r_1, r_2, ...$.

An example for a map $S$ where  Proposition~\ref{prop:adviadiffop} may be  applied,  is the normal map $r_n$ defined in Example~\ref{examp:rn}. In
particular, we obtain: $res_H(\pi_\lambda)$ is discretely decomposable if and only if there exists an $n$ so that the
image of $r_n$ is contained in $L^2(H\times_{\tau_n} Hom_\C ((\p/\p')^{(n)} ,W))_{disc}.$
	
 \begin{proof}[Proof of Lemma~\ref{prop:diffopiskern}]   We fix $\{ z_j, j=1,\cdots, \dim E \}, \{ w_i, i=1, \dots,
\dim W\} $   orthonormal basis for $E,W$ respectively.  Our hypothesis gives for every $f \in \,H^2(G,\tau), h\in H$ the
following equality holds\\ \phantom{xxxxxxxxxxxxxxxxx} $S(f)(h)= \sum_{j,b,i} d_{j,b,i}\, ([R_{D_b}f](h),w_i)_W\, z_j.$ \\ In \cite{At}, we find a
proof that in the $L^2-$kernel of an elliptic operator,  $L^2-$con\-ver\-gence  implies uniform convergence on compact sets  of the
sequence as well as any of its derivatives. Since the Casimir operator acts on $G/K$ as an
elliptic operator, the result on PDE just quoted applies to $\,H^2(G,\tau).$ Hence, the equality $(S(f)(h),z_j)_Z
=\sum_{b,i} d_{j,b,i} (R_{D_{b}}f (h),w_i)_W$ yields the map $ f \mapsto (S(f)(h),z_j)_Z$  is a continuous linear functional
on $\,H^2(G,\tau).$ Thus, owing to Riesz representation Theorem,  there exists a function $K_S : G\times H \rightarrow Hom_\C (W,E)$ so that  $a)$ holds. The hypothesis $S$ is an intertwining map, yields the
equality $K_S(h_1 yk, h_1hs)=\nu(s^{-1})K_S(y,h)\tau (k), h_1,h \in H, y \in G, s \in L, k \in K.$ The smoothness
for  $K_S$  follows from that,  for each $z \in E$, $K_S(\cdot ,\cdot)^\star z$ is equal  to the map  $(y,h)\mapsto h^{-1}y$ followed by the map
$x\mapsto K_S(x,e)^\star z$ and that $v\mapsto K_S(v,e)^\star z$ is an element of $\,H^2(G,\tau).$ Next, we justify the
four equalities in the computation below. The first equality  is due to the expression of $S$; the second is due to  the
identity $L_{\check D}(f)(e)=R_D (f)(e)$; the third is due to (\ref{eq:Klambda}); finally,  we recall for an arbitrary $D \in U(\g)$   any  smooth vector $f \in V_\lambda^\infty$ is in the domain for $L_D$, in particular,  $y \mapsto
K_\lambda(y,e)^\star w_i$ is in the domain for $L_D$.  These  four considerations justifies the following
equalities for any smooth vector $f \in \,H^2(G,\tau)  $,
	\begin{align*}
	\int_G (f(y), K_S(y,e)^\star z_j)_W dy & =(S(f)(e),z_j)\\
	& =\sum d_{j,b,i}(L_{\check D_b} f(e), w_i)_W \\
	& =\sum_{b,i} d_{j,b,i} \int_G (L_{\check D_b} f(y), K_\lambda(y,e)^\star w_i)_W dy \\
	& =\sum_{b,i}d_{j,b,i}\int_G ( f (y), L_{\check D_{b}^\star}^{(1)} K_\lambda(y,e)^\star w_i)_W dy.
	\end{align*}
	We observe, the first and last member of the above equalities defines continuous linear functionals on
$\,H^2(G,\tau)$ and the two functionals agree on a dense subspace, namely the subspace  of smooth vectors, whence \begin{equation*} K_S(y,e)^\star z_j
=\sum_{b,i} d_{j,b,i}L_{\check D_{b}^\star}^{(1)} K_\lambda(y,e)^\star w_i. \eqno{(\natural)}  \end{equation*}   Since, the right hand
side of the above equality is a $K-$finite vector for $\pi_\lambda,$  we have shown $b)$.  To show the  continuity
of $S$, we notice that $S$ is defined by the  Carleman kernel $K_S$  (for each $h\in H$, $K_S(\cdot,h)^\star z_j \in
\,H^2(G,\tau)$) and by hypothesis the domain of the integral operator defined by $K_S$ on $L^2(G\times_\tau W)$
contains $\,H^2(G,\tau).$ Since a Carleman kernel determines a closed map on its maximal domain, and $\,H^2(G,\tau)$
is a closed subspace,  we have $S : \,H^2(G,\tau) \rightarrow L^2(H\times_\nu E)$ is a closed linear map with domain
$\,H^2(G,\tau)$,  the closed graph Theorem leads us to the continuity of $S.$ \\ To show the converse statement, we
make explicit the hypotheses on $S$, \\ $K_S(hx,hh_1)=K_S(x,h_1), h,h_1 \in H, x \in G$; for each $z\in E$, $y \mapsto
K_S(y,e)^\star z$ is a $K-$finite vector in $\,H^2(G,\tau)$; domain of $S$ equal to $\,H^2(G,\tau)$. We show $S$ is continuous and $S$
the restriction of a differential operator. In fact, since $K_S(x,h)=K_S(h^{-1}x,e)$ and $K_S(\cdot,e)^\star z$ belongs to
$\,H^2(G,\tau)$ we obtain $K_S(\cdot,h)$ is square integrable and hence $S$ is a Carleman map. As in the direct
implication, we obtain $S$ is continuous. To verify $S$ is the restriction of a differential operator, we fix  a nonzero
vector $w \in W.$ Since $\,H^2(G,\tau)$ is an irreducible representation for $G$, a result of Harish-Chandra shows that the
underlying $(\g,K)-$module for $\,H^2(G,\tau)$ is $U(\g)-$irreducible. It readily follows that the function
$K_\lambda (\cdot,e)^\star w$ is nonzero (otherwise $K_\lambda$ would be the null function) and a $K-$finite vector, therefore, for each
$z_j$,  there exists $D_j \in U(\g)$ so that $K_S(\cdot,e)^\star z_j =L_{D_j} K_\lambda (\cdot, e)^\star w. $ For
$f$ smooth vector in $\,H^2(G,\tau),$  $h \in H$, we have \begin{multline*} S(f)(h^{-1})=S(L_{h}f)(e)   =\sum_j(S(L_{h}f)(e),z_j)z_j \\
=\sum_j \int_G (L_{h}f(y),K_S(y,e)^\star z_j)_W z_j   =\sum_j \int_G (L_{h}f(y),L_{D_j}K_\lambda (y,e)^\star w)_W
z_j  \\  =\sum_j \int_G (L_{D_j^\star}(L_{h}f)(y),K_\lambda (y,e)^\star w)_W dy z_j =\sum_j
(L_{D_j^\star}(L_{h}f)(e), w)_Wz_j \\ =\sum_j
((R_{\check{D_j^\star}})(L_h f)(e), w)_W z_j \\  =\sum_j
((R_{\check{D_j^\star}})(f)(h^{-1}), w)_W z_j. \end{multline*}  Thus, after we fix      a linear basis $\{R_i \}$
for $U(\g),$ and we write $\check{D_j^\star} =\sum_i d_{i,j}R_i$,  for any smooth vector $f \in \,H^2(G,\tau), h \in H$ we obtain \\ \phantom{xxxxxxxxxxxxxxxx} $S(f)(h )= \sum_{j,i }d_{i,j} ([
R_{R_i}(f)](h),w)z_j$. \\
	 We claim  the equality holds for every $f \in \,H^2(G,\tau).$  In fact,   each $f$ in $\,H^2(G,\tau)$ is limit of a sequence $f_n$ of smooth vectors. The result of PDE alluded to in the direct proof let us conclude that sequence of functions  $ \sum_{j,i }d_{i,j} ([
R_{R_i}(f_n)](h),w)z_j$ converges uniformly on compact sets to the right hand side evaluated at $f$. On the other hand, we already know $S$ is continuous, thus,   if necessary going to a subsequence, the
Riesz-Fischer Theorem yields the left hand side   converges pointwise  (a.e.) to $S(f)$.  Therefore, we have shown   $S$ is equal to  the restriction of a differential operator. Thus, we have  concluded a proof of
Lemma~\ref{prop:diffopiskern}.   \end{proof}
  \begin{proof}[ Proof of Proposition~\ref{prop:adviadiffop}.] For the vector spaces $E,W$, we fix orthonormal basis $\{z_j\}, \{w_i \}$ respectively.
Since $res_H(\pi_\lambda)$ is discretely decomposable,  Theorem~\ref{prop:symbrea} shows $S$ is an integral map and Remark~\ref{rmk:structS}   yields
$ K_S(\cdot,e)^\star z_j$ is a $L-$finite vector. The hypothesis  $res_H(\pi_\lambda)$ is an admissible
representation allows us to apply \cite[Prop. 1.6]{K3}. In this way we obtain that the subspace of $L-$finite
vectors in $\,H^2(G,\tau)$ is equal to the subspace of $K-$finite vectors.  Whence, $x \mapsto K_S(x,e)^\star z_j$
is a $K-$finite vector. By hypothesis,  $\,H^2(G,\tau)_{K-fin}$ is an  irreducible representation under the action
of $U(\g)$ and the function $y \mapsto K_\lambda (y,e)^\star w_i$ is nonzero and $K-$finite, hence, for each $i,j$
there exists $ C_{j,i} \in  U(\g) $ so that $[L_{C_{j,i}}^{(1)} K_\lambda ](y,e)^\star w_i = K_S(y,e)^\star z_j,\,
\text{ for \,all}\, y \in G.$ Therefore,  since $x \mapsto K_S(x,e)^\star z_j$ is a smooth vector for $G$, for $f
\in V_\lambda^\infty$  we  justify as in the proof of Lemma~\ref{prop:diffopiskern} the  fifth and seventh
equality in the following computation, the sixth equality is due to (\ref{eq:Klambda}), for $h \in H$,  we set $\tilde f(x)=L_{h^{-1}}(f)(x),$  the hypothesis of $S$ is a $H-$map justifies the first equality in \begin{align*} S(f)(h)=  & S(\tilde f)(e)
 =\sum_j (S(\tilde f)(e),z_j)_E z_j  \\ & =\sum_j \int_G (\tilde f(y), K_S(y,e)^\star z_j)_W dy z_j \\ & =\sum_j \int_G (\tilde f(y),
[L_{C_{j,i}}^{(1)} K_\lambda ](y,e)^\star w_i)_W dy z_j\\ & =\sum_j \int_G (L_{C_{j,i}^\star} \tilde f(y), K_\lambda
(y,e)^\star w_i)_W dy z_j \\ & =\sum_j ([L_{C_{j,i}^\star} \tilde f](e),w_i)_W z_j =\sum_j ([R_{\check{C_{j,i}^\star}}
\tilde f](e),w_i)_W z_j \\ &  =\sum_j  ([R_{\check{C_{j,i}^\star} } (f)](h),w_i)_W z_j
. \end{align*}   After we set  $D_{j,i}:= \check{C_{j,i}^\star} $ and we recall definition~\ref{eq:defdop} we
conclude that $S$ restricted to the subspace of smooth vectors agrees with the restriction of a differential
operator. In order to show the equality for generic    $f \in \,H^2(G,\tau)$ we argue as follows:  there
	exists a sequence $f_r$ of elements in $V_\lambda^\infty$ which converges in $L^2-$norm to $f$.  Owing to the
Casimir operator is elliptic on $G/K$,  the sequence $f_r$,  as well as any derivatives of the sequence,  converge
uniformly on compact subsets. Moreover,  owing to Harish-Chandra's Plancherel Theorem, we have that
$L^2(H\times_\nu E)_{disc}$ is equal to  a finite sum of eigenspaces for the Casimir operator for $\h.$ We know  the Casimir
operator acts as an elliptic differential operator on $L^2(H\times_\nu E)$, whence, $L^2(H\times_\nu E)_{disc}$ is a reproducing kernel subspace, in consequence,  point evaluation
is a continuous linear functional on $L^2(H\times_\nu E)_{disc}$ \cite{At}. Finally, the hypothesis on
$res_H(\pi_\lambda)$ gives the image of $S$ is contained in $L^2(H\times_\nu E)_{disc}$.  Therefore, we have
justified the steps in \begin{multline*} S(f)(h)=\lim_r S(f_r)(h) = \lim_r \sum_j ([R_{D_{j,i}} f_r](h), w_i)_W
z_j \\ = \sum_j ([R_{D_{j,i}} \lim_r f_r](h), w_i)_W z_j =\sum_j ( [R_{D_{j,i}} f](h), w_i)_W z_j. \end{multline*}
Whence, we have shown the first  statement in Proposition~\ref{prop:adviadiffop}.
	
	To follow we assume for some $\nu$ and  some nonzero intertwining $H-$map  $S :\,H^2(G,\tau) \rightarrow
L^2(H\times_\nu E )_{disc}$  is the restriction of a nonzero linear differential operator, we show  $res_H(\pi_\lambda)$
is discretely decomposable.\\
	The  hypothesis allows us to apply Lemma~\ref{prop:diffopiskern}. Thus,    $ y \mapsto
K_S(y,e)^\star z_j$ is a $K-$finite vector in $\,H^2(G,\tau).$ We claim, $K_S(\cdot ,e)^\star z_j$  is
$\z(U(\h))-$finite. In fact,   in a previous paragraph we argued that Harish-Chandra's Plancherel Theorem shows $L^2(H\times_\nu E )_{disc}$ is equal to a finite sum of
 eigenspaces of the Casimir operator for $\h$. Thus,  for $f \in V_\lambda^\infty, D \in \z(U(\h)) $ so that
 $S(f)$ belongs to  in an irreducible component of an  eigenspace for the Casimir operator, we have the equalities \begin{multline*}\int_G (f(y),L_{
D^\star}^{(1)}K_S(y,e)^\star z)_W dy=\int_G (L_D f (y), K_S(y,e)^\star z)_W  dy \\ = (S(L_D f) (e),z)_Z =\chi_\mu
(D) (S(f) (e), z)_Z \\=\chi_\mu(D)\int_G (f(y), K_S(y,e)^\star z)_W dy.   \end{multline*} The third equality holds
because by hypothesis $S(f)$ is an eigenfunction for $\z(U(\h)).$ Since,  each function $K_S(\cdot,e)^\star z, \,\, L_{
D^\star}^{(1)}K_S(y,e)^\star z$ belongs to $H^2(G,\tau)$, the first and last member of the
above equalities determine continuous linear functionals on $\,H^2(G,\tau)$, the two linear functionals agree in the dense subspace of
smooth vectors. Whence, $y\mapsto K_S(y,e)^\star z_j$ is an eigenfunction for $\z(U(\h)).$ The general case
readily follows. The hypothesis $S$ is nonzero and Remark 3.8 $b)$, gives us  $z \in E$ so that
$U(\h) K_S(\cdot,e)^\star z $ is  a $\z(U(\h))-$finite and  a nonzero $U(\h)-$submodule of $\,H^2(G,\tau)_{K-fin}. $  We quote a
result
	of Harish-Chandra: a $U(\h)-$finitely generated, $\z(U(\h))-$finite,  $(\h,L)-$module has a finite composition
series. For a  proof  (cf. \cite[Corollary 3.4.7 and Theorem 4.2.1]{Wa1}). Thus, $\,H^2(G,\tau)_{K-fin} $ contains an irreducible $(\h,L)$-submodule.    Therefore,   the subspace $U(\h)
K_S(\cdot,e)^\star z $ contains an irreducible $U(\h)-$sub\-mod\-ule. Next, in \cite[Lemma 1.5]{K3}  we find a proof
of: if a $(\g,K)-$mod\-ule contains an irreducible $(\h,L)-$submodule, then the $(\g,K)-$module is $\h-$al\-ge\-bra\-i\-cal\-ly
decomposable. Thus,  $res_H(\pi_\lambda)$ is  al\-ge\-bra\-i\-cal\-ly  discretely decomposable. The hypothesis $V_\lambda$ is
unitary yields  discrete decomposable \cite[Theorem 4.2.6]{K4}. Whence, we have shown $i)$.
	
	We now assume for some $(\sigma, Z)$ and every  intertwining linear $H-$map \\ $S :\,H^2(G,\tau) \rightarrow
H^2(H,\sigma)$ is the restriction of a linear differential operator. We show,   the multiplicity of $
(\pi_\mu^H, H^2(H,\sigma))$ in $res_H (\pi_\lambda)$ is finite.
	
	The hypothesis and $i)$ yields $res_H(\pi_\lambda)$ is discrete decomposable. Let's
assume the multiplicity  of $H^2(H,\sigma)$ in $res_H(\pi_\lambda)$ is infinite. Thus, there exists $T_1, T_2,
\dots $ so that $T_j : H^2(H,\sigma) \rightarrow \,H^2(G,\tau)[V_\mu^H]$ are isometric immersion    intertwining
linear $H-$maps so that  for $r\not= s$  the image of $T_r$ is orthogonal to the image of $T_s$ and the algebraic sum
of the subspaces  $T_r(H^2(H,\sigma)), \,\, r=1,2,\cdots $ is dense in $\,H^2(G,\tau)[V_\mu^H]$. Let $\iota : Z
\rightarrow H^2(H,\sigma)[Z]$ be the
	equivariant immersion adjoint to evaluation at the identity. We fix a norm one vector $g_0 :=\iota (z_0) \in
H^2(H,\sigma)[Z]$. There are two possibilities: For some $r$ the function $T_r(g_0)$   is not a
$K-$finite vector or else for every $r$ the function $T_r(g_0)$ is a $K-$finite vector. To
follow, we analyze  the second case, for this,  we  define $v_n :=T_n(g_0)$ and we choose a sequence of nonzero
positive real numbers $(a_n)_n$ so that $v_0:=\sum_n a_n v_n$ is not the zero vector.  Due to the orthogonality
for the image of the  $T_r$ and the choice of the sequence,   $v_0$ is not a $K-$finite vector.  Since the
stabilizer of $v_0$ in $H$ is equal to the stabilizer of $g_0$ on $H$, the correspondence  $T : H^2(H,\sigma)
\rightarrow \,H^2(G,\tau)$ defined by $T(h.g_0)=\frac{1}{\Vert v_0 \Vert} h.v_0$  extends to  an isometric
immersion. We claim $S=T^\star$ is not the restriction of a linear differential operator. For this, we show
$S(\frac{v_0}{\Vert v_0 \Vert})=g_0$ and $K_S(\cdot,e)^\star z_0 = \frac{v_0}{\Vert v_0 \Vert}$.  On one hand, we
have the following system of  equations  \\ \phantom{xxxxxxx} $ (S(f)(e),z_0)_Z = \int_G (f(y), K_S(y,e)^\star z_0)_W dy
\,\,\,\forall f \in \,H^2(G,\tau)$ \\ determine the function $K_S(\cdot,e)^\star z_0$.  On the other hand, for
arbitrary $f \in H^2(G, \tau)$ we have \begin{multline*}\int_G (f(y), \frac{v_0(y)}{\Vert v_0 \Vert})_W dy
=(f,T(g_0))_{H^2(G, \tau) } =(T^\star f, g_0)_{H^2(H,\sigma)} \\ =(S(f), \iota (z_0))_{H^2(H,\sigma)}
=(\iota^\star (Sf), z_0)_Z =(S(f)(e), z_0)_Z.\end{multline*} Thus,  we have shown the equality $K_S(\cdot,e)^\star
z_0 = \frac{v_0}{\Vert v_0 \Vert}.$ Therefore, if $S$ were a  differential operator, the fact
	$v_0$ is not a $K-$finite vector, would contradict  Lemma~\ref{prop:diffopiskern} $ b)$. In the first case,
a similar argument yields    $S:=T_r^\star $ is not the restriction of a linear differential operator.   This
concludes the proof of $ii)$ in Proposition~\ref{prop:adviadiffop}.  \end{proof}
\begin{cor} Let $\{ R_a\}$ be a linear basis for $U(\g)$. Then, there exists smooth functions $g_{a,b} $ on $H$ and complex numbers $d_{j,b,i}$ so
that  $$K_S (y,h)^\star z_j  =\sum_{a,b,i} d_{j,b,i} g_{a,b}(h) (L_{R_a}^{(1)} K_\lambda) (y,h)^\star w_i .$$
\end{cor}
In fact, we recall the equality $(\natural )$, let $g_{a,b}$ be smooth functions so that $Ad(h^{-1})\check{D_b^\star} =\sum_a g_{a,b}(h) R_a.$ We have: \begin{align*} K_S
(y,h)^\star z_j & = K_S (h^{-1}y,e)^\star z_j=\sum_{a,b,i} d_{j,b,i} (L_{\check{D_b^\star}}^{(1)} K_\lambda) (h^{-1}y,e)^\star
w_i \\ & =  \sum_{a,b,i} d_{j,b,i}  (L_{Ad(h^{-1})\check{D_b^\star}}^{(1)} K_\lambda) (y,h)^\star w_i \\ & = \sum_{a,b,i}
d_{j,b,i} g_{a,b}(h) (L_{R_a}^{(1)} K_\lambda) (y,h)^\star w_i. \end{align*}
\begin{examp}
	For a real form $H/L$ for the Hermitian symmetric space $G/K$ and a holomorphic discrete series $\,H^2(G,\tau)$
for $G$, any nonzero intertwining linear $H-$map  $S : \,H^2(G,\tau) \rightarrow H^2(H,\sigma)$ never is  the
restriction of a differential operator. The statement  holds because under our hypothesis $res_H(\pi_\lambda)$ is not discretely decomposable \cite{Ho}.
 \end{examp}
 \subsection{Equivariant maps into $K-$types}

We now show a   result  related to Lemma~\ref{prop:diffopiskern},  and by means of a  very similar proof. Let $(\vartheta, B)$ be an irreducible representation of $K$.  We fix a realization for $(\vartheta, B)$ as a subspace of the space of smooth sections for a bundle $ K\times_\theta C \rightarrow K/K_1$. Here, $K_1$ is a closed subgroup of $K$ and $(\theta, C)$ is a finite dimensional representation for $K_1.$
\begin{fact}\label{prop:gkisdiffop} Assume $(\vartheta,B)$ is a $K-$type for $\,H^2(G,\tau)$. Let $S : \,H^2(G,\tau)\rightarrow B$ be a continuous intertwining linear map. Then, $S$ is the restriction of a differential operator.
\end{fact}
Indeed, we fix $c \in C, k \in K$, then, the map $\,H^2(G,\tau) \ni f \mapsto (S(f)(k),c)_C$ is a continuous linear functional in $\,H^2(G,\tau)$. Thus, there exists $K_S : G\times K \rightarrow Hom_\C(W,C)$ so that $y\mapsto K_S(y,c)^\star c$  belongs to $\,H^2(G,\tau)$ and $(S(f)(k),c)_C =\int_G (f(y), K_S(y,k)^\star c)_W dy.$ Actually,  $y\mapsto K_S(y,c)^\star c$ belongs to $\,H^2(G,\tau)[B] \subset (V_\lambda)_{K-fin}.$ The $U(\g)-$irreducibility for $(V_\lambda)_{K-fin}$ implies there exists $D \in U(\g)$ so that  $L_D^{(1)} K_\lambda(\cdot,e)w = K_S(\cdot,e)^\star c$ holds. Recalling $K_S(y,k)=K_S(k^{-1}y,e)$ and proceeding as in the proof of Lemma~\ref{prop:diffopiskern} and its Corollary  Fact~\ref{prop:gkisdiffop} follows.
\begin{fact}\label{prop:gl1isdiffop}  Let $(\vartheta ,B)$ be as before. We assume $ res_H(\pi_\lambda, V_\lambda)$ is admissible. Then, any intertwining operator from $\,H^2(G,\tau)$ into $(\vartheta, B)$  is equal to  a restriction of a differential operator. \\ In fact, the $H-$admissibility forces $V_\lambda[B] \subset (V_\lambda)_{K-fin}$. Now, the proof goes word by word as the one for Fact~\ref{prop:gkisdiffop}.
\end{fact}

 \begin{rmk} In the setting of Fact~\ref{prop:gkisdiffop} or  Fact~\ref{prop:gl1isdiffop}, we further assume $B$ is realized as the kernel of differential operators. Then, any intertwining map $S :\,H^2(G,\tau) \rightarrow B$ extends to an intertwining map from the maximal globalization provided by the kernel of the Schmid operator into $B$. The extension is a differential operator.  The proof of this remark is as the proof for Theorem~\ref{thm:kobcon1}.
 	\end{rmk}

\begin{rmk} Any intertwining linear map $S :  \,H^2(G,\tau) \rightarrow L^2(H\times_\nu E) $ restricted to the subspace $\,H^2(G,\tau)_{H-disc}$ is at the same time the restriction of a Carleman kernel map and of a differential  operator.
\end{rmk}

\subsection{Extension of an intertwining map to  maximal globalization}

A conjecture of Toshiyuki Kobayashi   \cite{K6}  predicts     that  under certain hypothesis, each continuous intertwining linear operator between two maximal globalizations of Zuckerman modules, realized via Dolbeault cohomology,   is given by restriction of a holomorphic differential operator. In this subsection, we show
an analogous statement for the maximal globalization provided by  Schmid operators. \\ The symbols $G,K, (\tau,W), \,H^2(G,\tau), \pi_\lambda,  H, L, (\sigma, Z),   H^2(H,\sigma), \pi_\mu^H$ are as in Section~\ref{sec:prelim}. Let $$D_G : C^\infty (G\times_\tau W) \rightarrow C^\infty (G\times_{\tau_1} W_1)$$ be the Schmid operator  \cite{Sch} \cite{Wo}. Similarly, we have a Schmid operator $D_H : C^\infty (H\times_\sigma Z) \rightarrow C^\infty (H\times_{\sigma_1} Z_1).$ Since $D_G$   is an elliptic operator, $Ker(D_G)$ is a closed subspace of the space of smooth sections. Thus, $Ker(D_G)$ becomes a smooth Frechet representation $\ell$  for $G.$ Among the properties of the kernel of the operator $D_G$ are:  $\,H^2(G,\tau)$ is a linear subspace of $Ker(D_G);$ the inclusion map $\,H^2(G,\tau)$ into $Ker(D_G)$ is continuous;  the subspace of $K-$finite vectors in $Ker(D_G)$  is equal to the subspace of $K-$finite vectors for $\,H^2(G,\tau)$; $ Ker(D_G)$ is a maximal globalization for the underlying Harish-Chandra module for $(\pi_\lambda, \,H^2(G,\tau))$.   A similar statement holds for $D_H$. Now, we are ready to state the corresponding result.
\begin{thm} \label{thm:kobcon1} We assume $res_H(\pi_\lambda)$ is an $H-$admissible representation. Then, the following two statements hold: \\ a) Any continuous, $H-$intertwining   linear map $ S : Ker(D_G) \rightarrow Ker(D_H)$ is the restriction of a differential operator. \\ b) Any continuous $H-$intertwining linear map $S : \,H^2(G,\tau) \rightarrow H^2(H,\sigma)$ extends to a continuous intertwining operator from $Ker (D_G) $ to $Ker(D_H).$
\end{thm}
 In \cite[Theorem 3.6]{N}, Nakahama has shown a similar result under the hypothesis of both $G/K, H/L$ are Hermitian symmetric spaces, the inclusion $H/L $ into $G/K$ is holomorphic, and both representations are holomorphic discrete series

\begin{proof}[Proof of Theorem~\ref{thm:kobcon1}] We show $a)$.  Let $S$ be as in the hypothesis. Since each inclusion $H^2(G,\sigma) \subset Ker(D_G)$,  $H^2(H,\sigma) \subset Ker(D_H)$  is continuous, we have for $h \in H, z \in Z$ that the linear functional $\,H^2(G,\tau) \ni f  \mapsto (S(f)(h),z)_Z$ is continuous; whence,  Riez representation  Theorem implies there exists an element $y \mapsto K_S(y,h)^\star z $ of $\,H^2(G,\tau)$ so that $ (S(f)(h),z)_Z =\int_G (f(y), K_S(y,h)^\star z)_W dy . $ It readily follows that $y \mapsto K_S(y,e)^\star z$ is an $L-$fi\-nite vector. Since, for discrete series, the hypothesis of $H-$admissibility  implies $L-$admissibility, \cite{DV}, we apply  \cite[Proposition 1.6]{K3},  hence,   $y \mapsto K_S(y,e)^\star z$ is a $K-$finite vector. The $U(\g)-$irreducibility of the subspace of $K-$finite vectors yields $K_S(\cdot,e)^\star z =L_{D_{z,w}}^{(1)} K_\lambda (\cdot,e)^\star w$.  Next, for  a $K-$finite vector $f$  after a computation similar to the one in the proof of  Theorem \ref{prop:diffop}, we arrive    at the equality $(S(f)(h),z)_Z= ([R_{\check{D_{z,w}^\star}} f](h), w)_W$. The continuity of $S$ together with $Ker(D_G)$ is a maximal globalization, let us  conclude: $S$ is the restriction of a differential operator. Thus, we have shown $a)$. We now verify $b)$. The hypothesis of $H-$admissibility, let us apply Theorem \ref{prop:diffop}. Therefore, $S$ is the restriction of a differential operator. More precisely, $ S(f)(h)=\sum_{a,b,i} d_{a,b,i} ( [R_{D_b} f](h), w_i)_W z_a $ and $D_H(S(f))\equiv 0$ for   any $K-$finite vector $f$ in $\,H^2(G,\tau)$. We extend $S$ to $Ker(D_G)$ via the previous equality. Obviously the extension is continuous in smooth topology. We claim: the image  of the extension is contained in $Ker(D_H). $ Indeed, owing to the subspace of $K-$finite vectors in $\,H^2(G,\tau)$ is dense in $Ker(D_G)$ for  smooth topology,    we obtain $D_H(Sf)\equiv 0$ for every $f \in Ker(D_G)$. Whence, we have shown Theorem \ref{thm:kobcon1}.\end{proof}
Theorem~\ref{thm:kobcon1} is a step in the proof of  the conjecture of Kobayashi. Actually, the statement in  Theorem~\ref{thm:kobcon1} is a solution to the conjecture of Kobayashi if we choose as maximal globalizations the one  constructed via  Schmid operator.  In order to formulate the conjecture, we need to recall   notation as well as  results  from Schmid thesis   \cite{Sch}.
The Harish-Chandra parameter $\lambda$ gives rise to $G-$invariant complex structure on $G/T$, as well as a $K-$invariant complex structure on $K/T$ and holomorphic line bundles $\mathcal L_\lambda \rightarrow G/T$, $\mathcal L_\lambda \rightarrow K/T$ so that the representation of $K$ in $H^s(K/T,\mathcal O(\mathcal L_\lambda))$ is equivalent to $(\tau, W)$ and the representation $\ell^\star$  of $G$ on $H^s(G/T, \mathcal O(\mathcal L_\lambda))$ is infinitesimally equivalent to $(\pi_\lambda, V_\lambda^G).$ Here,  $s=\frac 12 dim K/T.$ Owing to the construction of the respective complex structures the inclusion map $i_K : K/T \rightarrow G/T$ is holomorphic. After we endow the space of smooth forms on $G/T$ with the smooth topology, the work of Schmid, Wolf and Hon-Wei Wong shows that the image of $\bar \partial $ is closed. Thus, $H^s(G/T, \mathcal O(\mathcal L_\lambda))$ affords a Frechet representation for $G.$ Next, we describe an equivalence $F_{K,T}$ between the representations $(\ell, Ker(D_G))$ and $(\ell^\star , H^s(G/T, \mathcal O(\mathcal L_\lambda))).$  For this, we model $(\tau, W)$ on $(\ell^\star , H^s(K/T, \mathcal O(\mathcal L_\lambda)))$. Then, for a smooth $(0,s)-$form $\varphi$ on $G/T$ with values on $\mathcal L_\lambda$,  Schmid associates the function $G \ni g \mapsto F_{K,T}(\varphi )(g):= i_K^\star( \ell^\star_g \varphi) \in H^s(K/T, \mathcal O(\mathcal L_\lambda)).$ Schmid shows that when $\varphi$ is closed we have that $F_{K,T}(\varphi)$ belongs to $Ker(D_G)$ and the resulting map from $H^s(G/T, \mathcal O(\mathcal L_\lambda))$ into $Ker(D_G)$ is bijective. We do not describe the  inverse of the map $F_{K,T}$.  Similarly, attached to  the representation $V_\mu^H$,  we have a holomorphic line bundle $\mathcal L_\mu$ over $H/U$,    and a map $F_{L,U} :H^{s'}(H/U, \mathcal O(\mathcal L_\mu)) \rightarrow Ker(D_H)$.
The  conjecture of Kobayashi  is: assume $res_H(\pi_\lambda)$ is an $H-$admissible representation. Then, for every discrete factor $H^2(H,\sigma)$ of $res_H(\,H^2(G,\tau))$, we have that any $H-$intertwining continuous map from $H^s(G/T,\mathcal O(\mathcal L_\lambda))$ into $H^{s'}(H/U, \mathcal O(\mathcal L_\mu))$ is a holomorphic differential operator. We are able to show that each  {\it intertwining operator is represented by a  differential operator}. In fact, let $S:  H^s(G/T,\mathcal O(\mathcal L_\lambda)) \rightarrow H^{s'}(H/U, \mathcal O(\mathcal L_\mu))$ be a continuous intertwining operator.  Then, $F_{L,U} S F_{K,T}^{-1}$ is a continuous intertwining linear map from $\,H^2(G,\tau)$ into $H^2(H,\sigma)$, Theorem~\ref{thm:kobcon1} yields that this  composition is a differential operator  in our sense.

\subsection{Comments on the relation among  Hom's}\label{subsec:relationamonghom}
As usual, $Hom_H(...,...)$ denotes the space of  continuous intertwining operators. We have the natural inclusions
\begin{multline*}
  Hom_{\h,L} (\,H^2(G,\tau)_{K-fin}, H^2(H,\sigma)_{L-fin}) \\ \supseteq  Hom_{H} (\,H^2(G,\tau)^{\infty}, H^2(H,\sigma)^\infty )  \\   \supseteq Hom_{H} (\,H^2(G,\tau), H^2(H,\sigma)), \end{multline*}
as well as, similar inclusions for $Hom's (H^2(H,\sigma), \,H^2(G,\tau)).$  We would like to point out that when $res_H(\pi_\lambda)$ is $H-$admissible the above inclusions are equalities.  In fact,
in  \cite[Lemma 1.3, Prop. 1.6]{K3}, it is shown that the above inclusions are equalities under the hypothesis of $res_H(\pi_\lambda)$ is  $L-$admissible. Actually, in  \cite{DV}, it is shown for discrete series representations  $H-$admissible is equivalent to be $L-$admissible. As a consequence, of both facts,  we obtain the equalities   \begin{equation*} \label{eq:L-fin=K-fin} \,H^2(G,\tau)_{L-fin} = \,H^2(G,\tau)_{K-fin}=\oplus_{M \in (\h,L)-irred} \,H^2(G,\tau)_{L-fin}[M]. \end{equation*} Here, the sum is algebraic. Thus, this work of Kobayashi together with Theorem~\ref{prop:diffop} shows that once we know a representation $V_\lambda$ is $H-$admissible, the associated branching law problem is algebraic.

 In different papers  T. Kobayashi, his co-authors and other authors have done a deep study of the space of continuous intertwining linear operators between two principal series representations; their results yields estimates for dimension of such spaces as well as precise computation of such spaces. Next, we present some comments of our work for discrete series representations.  For this paper,     $$ Diff_{H} (\,H^2(G,\tau), H^2(H,\sigma)) $$ is  the space of not necessarily continuous linear intertwining  maps that are restriction of differential operators. We have shown in Lemma~\ref{prop:diffopiskern}  the inclusion "automatic continuity Theorem"  $$ Diff_{H} (\,H^2(G,\tau), H^2(H,\sigma)) \subset Hom_{H} (\,H^2(G,\tau), H^2(H,\sigma)) $$
We have shown that a representation is  $H-$admissible if and only if for every $H^2(H,\sigma)$ the following equality holds,   $$  Diff_{H} (\,H^2(G,\tau), H^2(H,\sigma))= Hom_{H} (\,H^2(G,\tau), H^2(H,\sigma))  $$
Besides, we have shown that if for some $H^2(H,\sigma)$  we have $$ 0 < dim Diff_{H} (\,H^2(G,\tau), H^2(H,\sigma)),$$ then $res_H(\pi_\lambda)$ is discretely decomposable.  Therefore, $H-$discretely decomposable and multiplicity of $H^2(H,\sigma)$  in $res_H(\pi_\lambda )$ is infinite,  forces some symmetry breaking map from $\,H^2(G,\tau)$ into $H^2(H,\sigma))$ is not the restriction of a differential operator.

Also, for a closed subgroup $L_1$ of $K$, under the hypothesis that $\pi_\lambda$ is $L_1-$admissi\-ble, after we realize each irreducible representation $(\vartheta, B)$ for $L_1$ in some space of smooth sections, we have shown the equality\\ \phantom{xxxxxxxxxx} $Diff_{L_1} (\,H^2(G,\tau), B)=Hom_{L_1} (\,H^2(G,\tau), B).$ \\ Moreover,  each continuous $L_1$-map from $\,H^2(G,\tau)$ into $B$  extends to a $L_1-$dif\-fer\-en\-tial operator from $Ker(D_G)$ into $B.$

  \section{Two results on $L^2(G\times_\tau W)[V_{\lambda_1}^G]$}
 Besides,  $(\tau, W), \,H^2(G,\tau)$, as in section 2,   we consider another square   integrable irreducible representation $V_{\lambda_1}^G$ for $G$ of Harish-Chandra  parameter $\lambda_1$ and lowest $K-$type $(\tau_1, W_1).$ In this section, we show that any intertwining $G-$map from $H^2(G,\tau_1)$ into $L^2(G\times_\tau W)$ is an integral operator. In the second part of this section, we  compute a kernel for the orthogonal projector onto the isotypic component in $L^2(G\times_\tau W)$ determined by $V_{\lambda_1}^G$.
 \subsection{Analysis for the elements of  $Hom_G(H^2(G,\tau_1), L^2(G \times_\tau W))$} In this subsection, we develop a similar study  to the  one developed by  \cite{K5} on Shintani's functions. For a Harish-Chandra parameter $\lambda_1$  we study  continuous intertwining linear map  $T :H^2(G,\tau_1)=V_{\lambda_1}^G  \rightarrow L^2(G \times_\tau W)$.  To begin with, we show that the natural inclusion maps \begin{multline*} Hom_G(V_{\lambda_1}^G, L^2(G \times_\tau W)) \\ \subseteq Hom_G( (V_{\lambda_1}^G)^\infty , L^2(G \times_\tau W)^\infty) \\ \subseteq Hom_{(\g,K)}((V_{\lambda_1}^G)_{K-fin} , (L^2(G,\tau )^\infty)_{K-fin}). \end{multline*} are bijective. In fact, let $T_0$ be a linear map in $Hom_{(\g,K)}(...,...),$ then, due that  $Im T_0$ is contained in a Hilbert space, we have  $(closure(Im T_0))_{K-fin}=(Im T_0)_{K-fin}$. Thus, $closure(Im T_0)$ is an irreducible unitary representation. In \cite[Lemma 3.4.11]{Wa1}, we find a proof that $T_0$ extends to a continuous intertwining map $T$ from $H^2(G,\tau_1) $ onto the closure   of the image of $T_0.$  Hence, the claim follows.  Next, we show, for $T \in Hom_G(V_{\lambda_1}^G, L^2(G \times_\tau W))$  that
 \begin{prop}   $T$ is represented by a $G-$invariant smooth kernel $k_T$. We set $k(x):=K_T(x,1) . $ The function $k$ satisfies: $k : G \rightarrow Hom_\C (W_1,W) $ is  smooth; $ k(k_1 xk_2)= \tau(k_1) k(x) \tau_1 (k_2), k_1, k_2 \in K, x\in G$; $k$ is a solution to the equation $L_{\Omega_G} k =[ (\lambda_1, \lambda_1) -(\rho, \rho) ]k$;  and $k$ is square integrable.  Conversely,   given  a function $k$ which satisfies the four previous properties, then $Tf(x)=\int_G k(x^{-1}y) f(y) dy$ defines a continuous  intertwining continuous linear map from $H^2(G,\tau_1) $ into $L^2(G \times_\tau W).$
 \end{prop}

 \begin{proof}  For the direct affirmation we notice that  since $T$ is continuous,  Schur's Lemma yields that $T^\star T$ is a constant times the identity, hence, we may and will assume $T$ is an isometry into its image. Thus, $Im(T)$ is a closed irreducible left invariant subspace of $L^2(G \times_\tau W).$  Therefore,   $Im(T)$ is included in an eigenspace of the Casimir operator, whence $Im(T)$ is contained in the kernel of an elliptic $G-$invariant operator. Therefore,  $T$ is given by a   kernel $K_T : G\times G \rightarrow Hom_\C (W_1,W)$ so that the map $y \mapsto K_T(y,x)^\star  w $ belongs to $H^2(G, \tau_1)$ for each $w \in W, x \in G.$    Since $T$ is an intertwining map we have the equality $K_T(y,x)=K_T(x^{-1}y,e).$ Thus, $K_T$ is a smooth function. From $K_T (k y, x k_2)= \tau (k_2^{-1}) K_T(y,x) \tau_1(k)$ we obtain $k(k_1 x k_2)=\tau(k_1) k(x) \tau_1 (k_2).$ Since $ y \mapsto K_T(y,x)^\star  w$ belongs to $H^2(G,\tau_1)$ we obtain that $L_{\Omega_G} k = [(\lambda_1,\lambda_1)-(\rho,\rho)] k$ and that $k$ is square integrable. Conversely, given $k$ that satisfies the four properties listed. Then,  in
  \cite[Corollary to Lemma 65]{HC2},  it is shown $k$ is tempered in the sense of Harish-Chandra, the hypothesis $G$ is linear let us apply  \cite[Proposition 6]{OV} to $k $ and deduce  $k \in L^{2-\epsilon}(G, Hom_\C(Z,W))$ for $\epsilon$ small. Thus, the Kunze-Stein phenomena, \cite{C}, let us conclude that the proposed formula for $T$ defines a continuous linear map.    \end{proof}
  Actually, it  can be shown that the knowledge of one explicit immersion of the Harish-Chandra module for $H^2(G,\tau_1)$ into a non-unitary principal series representation yields that the kernel of each element in $Hom_G(H^2(G,\tau_1), \\ L^2(G \times_\tau W))$ can be written as a Eisenstein integral. \\
 One way to compute the  dimension of the space  $Hom_G(H^2(G,\tau_1), L^2(G \times_\tau W))$ is via Frobenius reciprocity and Blattner's formula \cite{DHV}. Thus, there is an explicit formula of  $dim \, Hom_G(H^2(G,\tau_1), L^2(G \times_\tau W))$ based on the Harish-Chandra parameter for $H^2(G,\tau_1)$, the highest weight for $W$ and a partition function associated to the noncompact roots with positive inner product with respect to  the Harish-Chandra parameter for $H^2(G,\tau_1).$
\subsection{Kernel for the projector onto $L^2(G\times_\tau W)[V_{\lambda_1}^G]$}\label{subsec:kernl2w} In this section we generalize a Theorem of \cite{WW} where they compute the kernel for  orthogonal projector onto the isotypic component $L^2(G\times_\tau W)[V_{\lambda}^G]=\,H^2(G,\tau).$ We also extend work of Shimeno \cite{Shi} for the case of line bundles over $G/K$.   We fix a representative $(\pi_{\lambda_1}, V_{\lambda_1}^G)$ of the representation for $G$ of Harish-Chandra parameter $\lambda_1$ and we assume $(\tau, W)$ is a $K-$type for $\pi_{\lambda_1}.$
We fix an orthonormal basis $\{f_j\}_{j=1\dots N} $ for $V_{\lambda_1}^G[W] $ and recall the spherical trace function $ \phi_1(z)=d_{\lambda_1} \sum_j (\pi_{\lambda_1}^G (z) f_j,f_j)_{V_{\lambda_1}^G}.$ Then,
\begin{prop} The orthogonal projector $P$ onto $L^2(G\times_\tau W)[V_{\lambda_1}^G]$ is kernel map given the "external kernel" $$  P(f)(x)=\int_G  \phi_1(x^{-1}y) f(y)dy \eqno{(\dag)}$$ and by the matrix valued Carleman kernel $$ (y,x) \mapsto \int_K \tau(k^{-1}) \phi_1 (x^{-1}yk) dk $$
\end{prop}
\begin{proof}

Let $Q$ denote the orthogonal projector onto $V_{\lambda_1}^G [W]$. Then, each nonzero intertwining $K-$map $ b: V_{\lambda_1}^G[W] \rightarrow W$ yields a map $f_b : V_{\lambda_1}^G \rightarrow L^2(G\times_\tau W)$ defined by  $ V_{\lambda_1}^G \ni v \mapsto (G \ni x \mapsto b(Q(\pi_{\lambda_1}(x^{-1}v)):=f_{b}(v, x)). $ It is obvious that $f_b$ gives rise to  an equivariant  embedding of $V_{\lambda_1}^G$ into $L^2(G\times_\tau W)[V_\mu^G]$. After we fix a linear basis  for $Hom_K(V_\mu^G[W], W), $
owing to Frobenius reciprocity, the subspace $L^2(G\times_\tau W)[V_{\lambda_1}^G]$ is equal to the linear span of the image of the functions $f_{b}$ when $b$ runs over the chosen linear basis for $Hom_K(V_{\lambda_1}^G[W], W) $. Next, the   spherical trace function $\phi_1$ is $K-$central, hence, it follows that for each $f$ in  $L^2(G\times_\tau W)$,  the right hand side of $ (\dag)$  belongs to $\Gamma (G\times_\tau W)$.  The hypothesis $G$ is a linear group let us conclude that $\phi_1 \in L^{2-\epsilon}(G)$ for some positive $\epsilon.$  Thus,  the Kunze-Stein phenomena \cite{C} yields the right hand side determines a continuous linear operator on $L^2(G\times_\tau W)$.  Next, as in Appendix~\ref{subsec:A6}, we verify  $\int_G  \phi_1(x^{-1}y) f_{b}(v,y)dy = f_{b}(v,x). $  Harish-Chandra Plancherel's Theorem yields the integral $(\dag)$ evaluated at a wave package   $f$ orthogonal to $L^2(G\times_\tau W)[V_\mu^G] \cap (L^2(G\times_\tau W))_{disc} $  is equal to zero. Thus, the first statement holds. The second statement  readily follows. \end{proof}

\section{Projector onto isotypic component of $res_H(\pi_\lambda)$}

\subsection{Projector onto isotypic components via differential operators}
\label{klambamuasderivative}

In \cite{N},      for a scalar  holomorphic discrete series, it is shown  that  the orthogonal projector onto an isotypic component is  equal to an  infinite order differential operator when restricted to the subspace of smooth vectors. This subsection aims is to analyze the general value  of his results. For this, we analyze ways of expressing the orthogonal projector onto an isotypic component using differential operators. To begin with, we study an example.

\begin{examp} Let $G=SU(1,1)$. We fix as $K=$diagonal matrices $(e^{i\varphi}, e^{-i\varphi}).$ The characters of $K$ are $e^{in\varphi}, n \in \Z.$ The discrete series $(\pi_\lambda, V_\lambda)$ is presented as a subspace of the space of holomorphic functions $f(z)=a_0 +a_1 z+\cdots +a_n z^n + \cdots $ on the unit disk $D=\{ z \in \C : \vert z \vert <1 \}$. The action of $G$ is by homographic transformations. The action of $K$ is such that a typical isotypic subspace is $V_\lambda[V_n^K]=\C z^n.$ It readily follows that $P_{V_\lambda[V_n^K]} (f)(z)= c_{\lambda, n} z^n \frac{\partial^n f}{\partial  z^n}(0)$ and $K_{V_\lambda[V_n^K]}(z,w)=d_{\lambda, n} \bar{z}^n w^n.$ Here, the reproducing kernel is $K_\lambda (z,w)=\frac{k_\lambda}{(1-\bar z w)^\lambda}$ and $K_{V_\lambda[V_n^K]}(z,w) = d_{\lambda,n} \bar{z}^n w^n \frac{\partial^n}{\partial  \bar {z}^n} \frac{\partial^n}{\partial {w}^n}(K_\lambda )(0,0)$. We point out that since $K_\lambda$ is a rational function and $K_{V_\lambda[V_n^K]}$ is a polynomial function, hence, no derivative of $K_\lambda$ is equal to $K_{V_\lambda[V_n^K]}.$  This example expresses   $P_{V_\lambda[V_n^K]} (f)(\cdot) = F_1(\cdot) D(f)(e)$, $ K_{V_\lambda[V_n^K]} (\cdot) = F_2(\cdot )\tilde{D} (K_\lambda)(e,e) $, where  $D,\tilde{D}$ are elements of $U( \mathfrak{sl}(2,\R))$, and, $F_1,F_2$ are smooth functions.  In Fact~\ref{prop:projontoktype}, we show this is a common feature for the orthogonal projector onto a $K-$isotypic component as  well as for its kernel. \end{examp}
\begin{examp} This example is a sequel to the previous one. We would  like to point out that if a differential operator   $D:=f_0 +f_1 \partial + \cdots +f_N \partial^N +\dots$ $(\partial =d/dz)$   is  equal to orthogonal projector onto $\C z^M$,  then, $D$ ought to be of infinite order. In fact, after a direct computation we obtain \begin{multline*}  D=\frac{Z^M}{M!}\partial^M -\frac{Z^{M+1}}{M!}\partial^{M+1}  +\frac{Z^{M+2}}{2 \,M!}\partial^{M+2} \\
+\dots +(-1)^p \frac{Z^{M+p}}{p! \,M!}\partial^{M+p}+ \dots
 \end{multline*}
\end{examp}
\begin{examp} Let $G,H,\,H^2(G,\tau)$ be as in Section~\ref{sec:prelim}. We assume $\pi_\lambda$ is an $H-$ad\-mis\-si\-ble representation. We fix an isotypic component $M$ for the action of $H.$ To follow, we   compute an infinite order differential operator $D$ so that its restriction to the subspace of  $K-$finite vectors agrees with the restriction of the orthogonal projector $P_M$ onto the isotypic component $M$. For this, we label the isotypic components as $M_1, M_2=M, M_3, M_4, \cdots $. We fix a nonzero element $\Omega_1$ in the center of $U(\h)$. Thus, $R:=L_{\Omega_1} $ acts by a constant $c_j$ on the subspace of $H-$smooth vectors in $M_j$. We further assume $c_k \notin \{ c_1, \dots, c_{k-1}\}$ for $k \geq 2$.   We claim: there exists a sequence of numbers $d_0, d_1, \dots $ such that $$ P_M= d_0 (R-c_1)+d_1 (R -c_1)(R-c_2)+ d_2 (R -c_1)(R-c_2)(R -c_3)+ \cdots $$
 In fact,  the hypothesis  $\pi_\lambda$ is $H-$admissible together  with a result of T. Kobayashi  implies that the subspace of $L-$finite vectors in $\,H^2(G,\tau)$ is equal to the subspace of $K-$finite vectors. Next, the hypothesis on $c_1 $ lets us find $d_0$ so that $d_0 (R-c_1)$ is equal to the identity in the subspace of $K-$finite vectors in $M_2$. Then, we find $d_1$ so that $d_0 (R-c_1)+d_1 (R -c_1)(R-c_2)$   is equal to zero on the subspace of $K-$finite vectors in $M_3$, and so on.
\end{examp}
The proof of the next fact shows a technique, which, let us compute  differential operators with smooth coefficients  to describe the orthogonal projector onto a finite dimensional  $K-$invariant subspace $M$ of $\,H^2(G,\tau)$.
\begin{fact}\label{prop:projontoktype}
	Let $(\vartheta, V_\vartheta^K)$ denote an irreducible representation for $K$. Let $P_{V_\lambda[V_\vartheta^K]}$
	denotes the orthogonal projector onto $V_\lambda[V_\vartheta^K]$, let $K_{V_\lambda[V_\vartheta^K] }: G\times G \rightarrow Hom_\C(W,W) $ denotes the kernel that represents $P_{V_\lambda[V_\vartheta^K]}$. We  fix an orthonormal basis $\{w_i\}$ for $W.$  Then, there exists a family
	$\{D_{i, \alpha}\} $ $ 1 \leq i \leq dim W, 1\leq \alpha \leq \dim V_\lambda[V_\vartheta^K] $ of elements in $U(\g)$ and complex valued smooth functions $F_{i,\alpha} $ on $G$ so that $$K_{V_\lambda[V_\vartheta^K] } (y,x)^\star w  =\sum_{i,\alpha} F_{i,\alpha} (x) (L_{D_{i,\alpha}}^{(1)}K_\lambda (y,e)^\star w, w_i)_W \,w_i.$$ Further, $$P_{V_\lambda[V_\vartheta^K]}(f)(x)=  \sum_{i,\alpha} \overline{ F_{i,\alpha} (x)} (L_{D_{i, \alpha}^\star}(f)(e), w_i)_W \,  w_i.$$
\end{fact}
Indeed, the identity $$(P_{V_\lambda[V_\vartheta^K]} (f)(x),w_i)_W =\int_G (f(y),K_{V_\lambda[V_\vartheta^K] }(y,x)^\star w_i)_W \, dy \eqno{(d)}$$ shows the function $y \mapsto K_{V_\lambda[V_\vartheta^K] }(y,x)^\star w_i$ is orthogonal to any function $f$ orthogonal to $V_\lambda[V_\vartheta^K].$ Thus, $y \mapsto   K_{V_\lambda[V_\vartheta^K] }(y,x)^\star  w_i$ belongs to $V_\lambda[V_\vartheta^K]$. The function $$G \ni x \mapsto H_i(x):= K_{V_\lambda[V_\vartheta^K] }(\cdot,x)^\star w_i \in V_\lambda[V_\vartheta^K] $$ is smooth, because  $V_\lambda[V_\vartheta^K]$  is a finite dimensional vector space and the equality $(d)$ shows that the evaluation of  coordinate functions in $H_i(x)$ yields a smooth function on $G$. Next, we set $k_w(\cdot) :=K_\lambda (\cdot ,e)^\star w$. Then, $k_w$ belongs to $V_\lambda [W]$,  and it is non zero for each nonzero $w$. We fix $i : 1 \leq i \leq \dim W.$  The $U(\g)-$irreducibility for $(V_\lambda)_{K-fin}$ implies there exists a finite dimensional vector subspace $N_{\vartheta, i}$ of $U(\g)$ so that the linear map  $R_i : N_{\vartheta,i} \rightarrow V_\lambda [V_\vartheta^K]$ defined by $R_i(D):=L_D (k_{w_i}) $ is bijective. Therefore, the composition $ R_i^{-1} H_i$ is a smooth function on $G$. To follow, we fix a linear basis $D_{i,\alpha} , 1 \leq \alpha \leq \dim V_\lambda[V_\vartheta^K]$ for $N_{\vartheta, i}$. Hence, $ R_i^{-1} H_i (x)= \sum_{\alpha} F_{i,\alpha}(x) D_{i,\alpha}$ where $F_{i,\alpha}$ are complex valued smooth functions on $G.$ Thus, \begin{multline*} K_{V_\lambda[V_\vartheta^K] }(y,x)^\star w_i =H_i(x)(y)=R_i ( R_i^{-1} (H_i(x)))(y)\\ = \sum_\alpha F_{i,\alpha} (x) L_{D_{i,\alpha}}^{(1)}(K_\lambda)(y,e)^\star w_i \, \forall y,x \in G. \end{multline*}
\begin{multline*}K_{V_\lambda[V_\vartheta^K] }(y,x)w=\sum_i (K_{V_\lambda[V_\vartheta^K] }(y,x)w,w_i)_W w_i \\ =\sum_i (w,K_{V_\lambda[V_\vartheta^K] }(y,x)^\star w_i)_W w_i  = \sum_{i,\alpha} (w, F_{i,\alpha}(x) L_{D_{i,\alpha}}^{(1)}K_\lambda (y,e)^\star w_i)_W w_i \\ =\sum_{i,\alpha} (\overline{F_{i,\alpha}(x)} L_{\bar D_{i,\alpha}}^{(1)}K_\lambda (y,e)w, w_i)_W w_i.
\end{multline*}
For $f \in V_\lambda^\infty$, the following equalities hold.
\begin{multline*} P_{V_\lambda[V_\vartheta^K]} (f)(x) \\=\sum_i (P_{V_\lambda[V_\vartheta^K]} (f)(x),w_i)_W w_i =\sum_i \int_G (f(y,K_{V_\lambda [V_\vartheta^K]}(y,x)^\star w_i)_W dy w_i \\ =\sum_i \int_G (f(y),\sum_\alpha F_{i,\alpha}(x) L_{D_{i,\alpha}}^{(1)}K_\lambda (y,e)^\star w_i)_W dy w_i \\ =\sum_i \int_G (\sum_{\alpha} \overline{F_{i,\alpha}(x)} L_{D_{i,\alpha}^\star}(f)(y),  K_\lambda (y,e)^\star w_i)_W dy w_i \\ =\sum_i ( \sum_{\alpha} \overline{F_{i,\alpha}(x)} L_{D_{i,\alpha}^{\star}}(f)(e), w_i)_W  w_i.
\end{multline*}
Finally, the equality of the first and last member extends to any  $f \in H^2(G,\tau)$
  owing to that,  for a fixed $x \in G$,  both members are continuous linear functionals on $H^2(G,\tau)$ and they agree on smooth vectors.   This concludes the proof of Fact~\ref{prop:projontoktype}.
\begin{fact} \label{prop:projectorviado}
	A completely similar result to Fact~\ref{prop:projontoktype} holds after we replace $V_\lambda[V_\vartheta^K]$ by a finite dimensional subspace $M$ of $(V_\lambda)_{K-fin}.$ Of course, the elements $F_{i,\alpha}, D_{i, \alpha}$   depend on $M$ and they  are highly non unique.
\end{fact}
On the basis  of Fact~\ref{prop:projontoktype}, Fact~\ref{prop:projectorviado}, we say that a linear map $T$  from a subspace of $\,H^2(G,\tau)$ into  $\,H^2(G,\tau)$ is expressed {\it by means of  differential operators} if for each orthonormal basis $\{w_i \}$ for $W$, it is possible to find finitely many  smooth complex valued functions $f_{i,\alpha}$ on $G$ and a finite family $\{D_{i,\alpha}\}$ of elements in $U(\g)$ so that for every $f$ in the domain of $T$ and $x \in G$  we have $T(f)(x)=\sum_{i,\alpha} f_{i,a}(x) (L_{D_{i,\alpha}}(f)(e), w_i)_W w_i.$

The next result shows a relation  between discretely decomposable and orthogonal projectors represented by means of   differential operators.
\begin{prop}\label{prop:disdecviaprodo} We assume $res_H(\pi_\lambda)$ is an $H-$admissible representation. Then, for any finite dimensional subspace $M$ of $(V_\lambda)_{L-fin}$ the orthogonal projector onto $M$ is expressed  by means of differential operators. Conversely, if for some Harish-Chandra parameter $\mu$ for $H$,  the orthogonal projector onto some nontrivial finite dimensional   subspace of $(V_\lambda[V_\mu^H])_{L-fin}$ is expressed by means of  differential operators, then,  $res_H(\pi_\lambda)$ is discretely decomposable. Furthermore, if for each nontrivial finite dimensional $L-$invariant subspace of $V_\lambda[V_\mu^H]$, its orthogonal projector is expressed by means  of differential operators, then, the multiplicity of $V_\mu^H$ in $res_H(\pi_\lambda)$  is finite.
\end{prop}
\begin{proof} For the direct implication, the hypothesis $H-$admissible let us  apply \cite[Proposition 1.6]{K3}. Thus, $(V_\lambda)_{L-fin}=(V_\lambda)_{K-fin}.$ Whence, Fact~\ref{prop:projectorviado} yields there exists
	$\{ D_{i,\alpha} \}_, 1\leq i \leq \dim W,  1 \leq \alpha \leq  \dim M $
	of elements in  $U(\g)$ and complex valued smooth functions $F_{i,\alpha}$ on $G$, so that for every $f \in V_\lambda$ we have $$P_M(f)(x) =\sum_{i,\alpha} \overline{F_{i,\alpha}(x)}(L_{D_{i,\alpha}^\star} f (e),w_i)_W w_i.$$ Whence, $P_M$ is computed by means of differential operators. For the converse statement, the hypothesis on $M$ gives   an expression for $P_M$ as above and for $K_{P_M}$ ($w,v \in W$) we have $$(K_{P_M}(y,x)^\star w,v)_W= \sum_{i,\alpha} F_{i,\alpha}(x)   (L_{D_{i,\alpha}}^{(1)}K_\lambda (y,e)^\star w, w_i)_W ( w_i,v)_W.$$ Hence,  for a fixed $x \in G$, the function $y \mapsto K_{P_M}(y,x)^\star w$ is $K-$finite, because the above expression gives that  the function $y \mapsto K_{P_M}(y,x)^\star w$ is equal to a finite sum of     $K-$finite vectors of type $y\mapsto L_D K_\lambda(y,e)^\star w. $
	Now, any $H-$smooth element in $V_\lambda [V_\mu^H]_{L-fin}$ is $\z(U(\h))-$finite, and, by hypothesis, $M \subset V_\lambda [V_\mu^H]_{L-fin}$. Thus, $K_{P_M}(\cdot ,x)^\star w$ is  $K-$finite and  $\z(U(\h))-$finite. Whence, as in previous Theorems, we conclude $res_H(\pi_\lambda)$ is algebraically decomposable. Thus, $res_H(\pi_\lambda)$ is discretely decomposable. To follow, we show the last statement in  Proposition~\ref{prop:disdecviaprodo}. The proof goes parallel to the proof of Theorem~\ref{prop:diffop} ii). That is, we assume $V_\lambda[V_\mu^H] $ is not $H-$admissible and we built up a finite dimensional $L-$invariant subspace $M$ so that $P_M$ is not expressed by means of  differential operators. Let $T_j :V_\mu^H \rightarrow V_\lambda[V_\mu^H] , $ $j=1,\dots$  be isometric immersions  $H-$maps so that the  subspaces $T_j(V_\mu^H)$ are pairwise orthogonal. We fix $v_0 \in  V_\mu^H[V_\sigma^L ]$ of norm one. There are two possibilities, either every $T_j(v_0)$ is a $K-$finite vector, or at least one $T_j(v_0)$ is not a $K-$finite vector. In the second  case we set $w:=T_j(v_0) $, and  $M=\text{linear \,span}\, \pi_\lambda (L)w$. In the first case, we choose a sequence of positive real numbers $(a_n)_n $ so that $w:=\sum_n a_n v_n$ is nonzero and we set $M=\text{linear \,span}\, \pi_\lambda (L)w$. Given that $\pi_\lambda$ is $K-$admissible, we have $w$ is  not a $K-$finite vector. As in the proof of Theorem~\ref{prop:diffop}, $M$ is a finite dimensional subspace of $V_\lambda[V_\mu^H]$ and $P_M$ is not the restriction of a differential operator because $K_{P_M}(\cdot,e)^\star v_0=w $ is not a $K-$finite vector.  This concludes the proof of Proposition~\ref{prop:disdecviaprodo} \end{proof}
 Next,  after we assume  $res_H(\pi_\lambda)$ is an $H-$admissible representation, we show the orthogonal projector onto an isotypic component for $H$  can be expressed by means of  an infinite degree differential operator on the subspace of smooth vectors for $G$. The result  generalizes   \cite[Theorem 3.10]{N} for $\pi_\lambda$ a scalar holomorphic discrete series.\\
The set up for the next Theorem   $ G,K,  (\tau, W),  \,H^2(G,\tau)=V_\lambda, \pi_\lambda, H,L,  V_\mu^H \\ = H^2(H,\sigma), \pi_\mu^H $  are as in Section~\ref{sec:prelim}.
\begin{thm} \label{prop:projlamuisdo} We assume $res_H(\pi_\lambda)$ is an admissible representation for $H.$ We denote by $\sigma_1, \sigma_2, \cdots $ the $L-$types for $V_\mu^H$.  Let $P_{\lambda, \mu}$ denote the orthogonal projector on the isotypic component $V_\lambda [V_\mu^H]$. Then, for $ j \in \mathbb N$, $1\leq i \leq \dim W$,    $1 \leq \alpha \leq \dim V_\lambda[V_\mu^H][V_{\sigma_j}^L]$,  there exists a family $\{ D_{j,i,\alpha} \},$   of elements in $U(\g)$ and complex valued smooth functions $F_{j,i, \alpha}$ on $G$,  such that for each smooth vector $f \in V_\lambda$ we have $$P_{\lambda,\mu} (f)(x)=\sum_{j\in \mathbb N} \sum_{ i, \alpha} F_{j,i,\alpha}(x) (L_{D_{j, i,\alpha}} (f)(e),w_i)_W\,\, w_i . $$  The convergence is  in smooth topology. Furthermore, for a $K-$finite function  $f$, the sum on the right is finite.
\end{thm}
\begin{proof} In \cite{DV}, we find a proof that the hypothesis $H-$admissible implies $V_\lambda$ is also $L-$admissible. Thus, our hypothesis leads that  the isotypic subspaces $V_\lambda[V_\mu^H][V_{\sigma_j}^L]$ are nonzero, finite dimensional and we have the Hilbert sum $V_\lambda[V_\mu^H]=\oplus_j V_\lambda[V_\mu^H][V_{\sigma_j}^L]$. Furthermore, the hypothesis of being $H-$ad\-mis\-si\-ble forces all the $L-$finite vectors are $K-$finite vectors \cite{K3}, hence, each subspace $V_\lambda[V_\mu^H][V_{\sigma_j}^L]$ is contained in $(V_\lambda)_{K-fin}.$ We now apply  Proposition~\ref{prop:disdecviaprodo} to obtain finitely many elements $D_{j,i,\alpha}, i=1,\cdots, \dim W, 1 \leq \alpha \leq \dim V_\lambda[V_\mu^H][V_{\sigma_j}^L] $ in $ U(\g)$ so that \\ \phantom{xxxxxx} $P_{V_\lambda[V_\mu^H][V_{\sigma_j}^L]} (f)(x)= \sum_{i, \alpha} F_{j,i,\alpha} (x) (L_{D_{j,i,\alpha}} (f)(e), w_i)_W \, w_i.$ \\ Next, the series $\sum_j P_{V_\lambda[V_\mu^H][V_{\sigma_j}^L]} $ converges pointwise to $P_{\lambda, \mu}$. Further, in \cite{HC2}, we find a proof that for  a smooth vector $f$, the convergence is absolute  in the smooth topology.  Whence, we have obtained the first statement in Theorem~\ref{prop:projlamuisdo}. The second statement follows because $L-$admissible implies each isotypic component for an irreducible representation of  $K$ is contained in a finite sum of isotypic components for $L.$
\end{proof}

\subsection{Kernel for the projector onto $\,H^2(G,\tau)[V_\mu^H]$}\label{sec:kerh2hw}

  Let $G,H, (\tau, W)$, $\\ \,H^2(G,\tau), K_\lambda $ be as  in Section~\ref{sec:prelim}. Let $(\pi_\mu^H,V_\mu^H)$ denote an irreducible square integrable representation for $H.$ Let $\,H^2(G,\tau)[V_\mu^H]$ denote  isotypic component for $V_\mu^H$ in $\,H^2(G,\tau)$ (cf. Notation).  Since, $\,H^2(G,\tau)$ is a reproducing kernel space, we have that $\,H^2(G,\tau)[V_\mu^H]$ is a reproducing kernel space.    Thus, the orthogonal projector $P_{\lambda ,\mu}$ onto $\,H^2(G,\tau)[V_\mu^H]$ is a represented by a Carleman matrix kernel $K_{\lambda, \mu}$. In this section, under certain hypothesis,  we  express the matrix  kernel $K_{\lambda, \mu}$ in terms of the matrix   kernel $K_\lambda$ and the distribution character $\Theta_{\pi_\mu^H} $ of the representation $(\pi_\mu^H, V_\mu^H). $   We are quite convinced the formula is true under more general hypothesis. The proposed formula is
\begin{prop}\label{prop:klmchar} Assume the restriction to $H$ of $\pi_\lambda$ is an $H-$admissible representation. Then, $P_{\lambda, \mu}$     is equal to the Carleman  operator given by the kernel $$ (y,x)\mapsto d_\mu \Theta_{\pi_\mu^H} (h \mapsto K_\lambda (h^{-1}y,x) )=d_\mu \Theta_{(\pi_\mu^H)^\star}(h \mapsto K_\lambda (hy,x)).$$
\end{prop}

  In order to avoid cumbersome notation, for this subsection, sometimes, we write  $xv:=\pi_\lambda (x)v$, $ x\in G, v \in \,H^2(G,\tau).$
A proof of  Proposition \ref{prop:klmchar} is given at the end of this section.  For the time being, we show Proposition \ref{prop:klmchar}  under the hypothesis:  $G$ is any Lie group and  $K,H$ are compact subgroups of $G$. Thus, $V_\mu^H$ is a finite dimensional vector space.  We fix $N$  a reproducing kernel   $G-$invariant subspace of $L^2(G \times_\tau W).$  Under these hypotheses, the orthogonal projector $P_N$ onto $N$ is represented by different  kernels. For any
    kernel $K_N$ that represents the orthogonal projector $P_N$,  we want to show $P_{N[V_\mu^H ]} $ is represented by $K_1(y,x):= \int_H  d_\mu \bar \chi_{ \pi_\mu^H} (h) K_N (h^{-1}y,x) dh$.
According to a classical result, the orthogonal projector onto $L^2(G\times_\tau W)[V_\mu^H]$ is the linear operator $"\pi (d_\mu \bar{\chi}_{ \pi_\mu^H} )".$ Therefore, for $f \in N$ and $x \in G$, we have
 \begin{align*}
      P_{N[V_\mu^H]} (f)(x) & =  "\pi (d_\mu \bar{\chi}_{ \pi_\mu^H} )"(P_N f)(x) \\
       & =\int_H d_\mu \bar{\chi}_{ \pi_\mu^H} (h) P_N(f)(h^{-1}x) dh \\
        & =\int_H d_\mu \bar{\chi}_{\pi_\mu^H} (h)\int_G K_N (y,h^{-1}x)  f(y) dy dh \\
        & = \int_G (\int_H  d_\mu \bar{\chi}_{ \pi_\mu^H} (h) K_N (y,h^{-1}x) dh) f(y) dy \\
        & = \int_G K_1 (y,x) f(y)dy .
     \end{align*}  Thus, $K_1 $ is a kernel that represents the orthogonal projector onto  $N[V_\mu^H].$ This concludes the verification of Proposition \ref{prop:klmchar} for  a compact subgroup  $H$.

\smallskip

Our proof of  Proposition \ref{prop:klmchar} is based on  an expression for the orthogonal projector onto a closed subspace  $E$ of $\,H^2(G,\tau)$. For this, we fix a representation $(\pi, V)$  equivalent  to $(\pi_\lambda, \,H^2(G,\tau))$ and we assume $W \subset V.$ Then, the map $ V \ni v \mapsto (G \ni x \mapsto f_v(x):=P_W( \pi(x^{-1}) v)$ is $G-$equivariant, continuous and bijective from $V$ onto $\,H^2(G,\tau)$. We also fix an orthonormal basis $\{v_i \}$ for $W$,  the  equality\\  \phantom{xxxxxxxxxxxxxx} $P_W ( \pi(y^{-1}) v)=\sum_{1 \leq j \leq dim W}( \pi(y^{-1}) v, v_j)_V v_j$\\ shows any element of $\,H^2(G,\tau)$ is a finite sum of matrix coefficients for $V$. \\
We notice, for $z,v \in V$,  $(f_v,f_z)_{L^2(G)} = \frac{\dim W}{d_\lambda} (v,z)_V $. Hence, an unitary equivalence $i$ from $W$ onto $\,H^2(G,\tau)[W]$ is given by $i(v)=\sqrt{\frac{d_\lambda}{\dim W}} f_v,$ we have  $e_1(i(v))= i(v)(e) = \sqrt{\frac{d_\lambda}{\dim W}} v.$ Hence $f_j := i(v_j)=\sqrt{\frac{d_\lambda}{\dim W}} f_{v_j}$ is an orthonormal basis for $\,H^2(G,\tau)[W].$ As usual, $\{v_j^\star \}$ denotes the dual basis to the basis $\{ v_j\}.$\\
For $f \in \,H^2(G,\tau)$  the integral below is absolutely convergent, because the product of two $L^2$ functions gives an integrable function. We define\\  \phantom{xx}  $C_E(f)(x) :=\frac{d_\lambda}{\dim W}\int_G \sum_{j,k} (P_E (yf_j), xf_k)_{L^2(G)} (v_k \otimes v_j^\star) (f(y))dy.  $

 The indexes  in the sum run from 1 to $\dim W$.   A straightforward computation, shows that, for $f \in L^2(G \times_\tau W)$,   the function $C_E(f)$  belongs to $\Gamma(G \times_\tau W).$   We want to show, \begin{lem}\label{lem:ce}  $H^2(G,\tau)$ is  as in Section~\ref{sec:prelim}, $E$ is a closed subspace of $H^2(G,\tau)$. Then,  $C_E$ is equal to the orthogonal projector from $\,H^2(G,\tau)$ onto $E.$ \end{lem}
   \begin{proof}
For $f_v  \in E \, (\text{resp.} \,\, f_v \in E^\perp),$ we show that $C_E(f_v)= f_v    (\text{resp.} \\ C_E(f_v)=0 ). $ In fact, $f_v(y)=\sum_{1 \leq r \leq \dim W} (\pi (y^{-1})v, v_r)_{V} v_r .$ Hence,
\begin{align*} \frac{\dim W}{d_\lambda} C_E(f_v)(x)& =\sum_{k,j,r}\int_G (P_E(yf_j),xf_k)_{L^2(G)} (\pi (y^{-1})v, v_r)_{V} v_j^\star(v_r) \,v_k \,dy
\\  &= \int_G \sum_{j,k}\ (yf_j, P_E(xf_k))_{L^2(G)} \overline{ (\pi (y)v_j, v)}_{V} dy\,\, v_k \\ &= \frac{d_\lambda}{dim W} \int_G \sum_{j,k}\ (yf_j, P_E(xf_k))_{L^2(G)} \overline{ (yf_{v_j}, f_v)}_{L^2(G)} \, dy \,\, v_k \\  & =
\frac{1}{ dim W  } \sum_{j,k} (f_{j},f_{v_j})_{L^2(G)} \overline{(P_E(xf_k), f_v)}_{L^2(G)} v_k \\
& =\frac{1}{dim W} (\sum_j (f_{j},f_{v_j})_{L^2(G)})( \sum_k (  P_E(f_v), xf_k)_{L^2(G)} v_k)  .
\end{align*}
Whence,  for $ f_v\in E^\perp $, we have $C_E(f_v)=0$, whereas for $f_v \in E$,  since $P_E(f_v)=f_v$, we obtain  \begin{multline*}\sum_k (P_E(f_v), x f_k)_{L^2(G)} v_k    = \sqrt{\frac{d_\lambda}{\dim W} }\sum_k (x^{-1}f_v, f_{v_k})_{L^2(G)} v_k  \\ =\frac{\sqrt{\dim W}}{\sqrt{d_\lambda}} \sum_k(\pi (x^{-1})v,v_k)_V v_k = \frac{\sqrt{\dim W}}{\sqrt{d_\lambda}} f_v(x) .\end{multline*}
   And, $ \sum_j (f_{j},f_{v_j})_{L^2(G)} =\sum_j \frac{\sqrt{d_\lambda}}{\sqrt{\dim W}}   (f_{v_j},f_{v_j})_{L^2(G)} = \frac{\sqrt{\dim W}}{\sqrt{d_\lambda}} \dim W. $  Thus, $\frac{\dim W}{d_\lambda} C_E(f_v)= \frac{\dim W}{d_\lambda}  f_v,$ and we have shown  Lemma~\ref{lem:ce}.
\end{proof}
{\it Note.} A consequence of  Lemma~\ref{lem:ce} is that $C_E$ is a continuous linear operator  in $\,H^2(G,\tau)$. Further, it readily follows that $C_E$ is a continuous linear operator on $L^2(G\times_\tau W)_{disc}$.
\begin{proof}[Proof of Proposition~\ref{prop:klmchar}]
After we recall the equality \\ \phantom{xxxxxxx} $K_\lambda (y,x) = e_1 \circ P_{\,H^2(G,\tau)[W]} \pi_\lambda (x^{-1}y)P_{\,H^2(G,\tau)[W]} \circ i$\\ and Lemma~\ref{lem:ce},  we conclude:
to show that the matrix kernel for the orthogonal projector   onto  $\,H^2(G,\tau)[V_\mu^H ]$ is equal to  the function $(y,x)\mapsto d_\mu \Theta_{V_\mu^H}(h \mapsto K_\lambda (h^{-1}y,x)),$  is equivalent to show the equality
    \begin{multline*}  \sum_{i,j} d_\mu \Theta_{(\pi_\mu^H)^\star} (h \mapsto (yf_j, xf_i)_{L^2(G)}) \, v_i\otimes v_j^\star   \\  = \sum_{i,j} (P_{\lambda, \mu}(yf_j), xf_i)_{L^2(G)}\, v_i \otimes v_j^\star.    \end{multline*}
The right hand side of the above equality obviously is  a well defined function, we now show the left hand side defines a  function. For this, we show that for fixed $y,x$ the function $H \ni h\mapsto K_\lambda(hy,x)$ is tempered in the sense of Harish-Chandra. Indeed, the $f_j's$ are $K-$finite vectors, hence    $\pi_\lambda (x)f_j, \pi_\lambda (y)f_j$ are smooth vectors for $\,H^2(G,\tau) $. Thus, they are tempered functions in the sense of Harish-Chandra. Since, the inner product $(\pi_\lambda (h) \pi_\lambda (x) f_i, \pi_\lambda (y) f_j)_{L^2(G)}$ can be rewritten as a convolution of tempered functions, we obtain that when we let $h$ varies in $G,$ the matrix coefficient $(\pi_\lambda (\cdot) \pi_\lambda (x) f_i, \pi_\lambda (y) f_j)_{L^2(G)}$  is a tempered function  on $G$. In     \cite[Proposition 2.2]{HHO}, it is shown that the restriction to $H$ of $(\pi_\lambda (\cdot) \pi_\lambda (x) f_i, \pi_\lambda (y) f_j)_{L^2(G)}$ is a tempered function. Since the character of a discrete series representation is a tempered distribution, we obtain  the left hand side  defines a function of $x,y.$ To follow, we fix $x, y \, \text{in} \, G$,  smooth vectors $v,w \, \text{in}\, \,H^2(G,\tau).$  We  verify
\begin{multline*}
\Theta_{(\pi_\mu^H)^\star} (h \mapsto (\pi_\lambda (hy)v, \pi_\lambda (x) w)_{L^2(G \times_\tau W)} ) \\  =(P_{\lambda, \mu} (\pi_\lambda (y)v), P_{\lambda, \mu} (\pi_\lambda (x)w))_{L^2(G, \tau)}  \\  =(P_{\lambda, \mu} (\pi_\lambda (y)v), \pi_\lambda (x)w)_{L^2(G, \tau)}. \end{multline*}
{\it Scholium:} We have that a $H-$smooth vector in $V_\lambda[V_\nu^H]$ is $G-$smooth. In fact, by hypothesis, $ V_\lambda[V_\nu^H]$ is a finite sum of irreducible unitary  representations for $H.$ Thus, \cite[Theorem 11.8.2]{Wa2}, the subspace  $(V_\lambda[V_\nu^H])^{H-smooth}$ is finitely generated module over the algebra   of rapidly decreasing functions $\mathscr S(H)$ on $H.$ Hence, each element of $(V_\lambda[V_\nu^H])^{H-smooth}$ is a finite sum of functions $(\pi_\lambda)_{\vert_H} (g) (f_1)$,  where $g$ is a rapidly decreasing function on $H$ and $f_1$ is an $L-$finite element in $V_\lambda[V_\nu^H])$. Now, from \cite{DV} \cite[Proposition 1.6]{K3} it follows that the hypothesis of being $H-$admissible yields that $L-$finite elements are $K-$finite. Thus, any vector in $(V_\lambda[V_\nu^H])^{H-smooth}$ is a finite sum of functions of the type  $(\pi_\lambda)_{\vert_H} (g) (f_1)$  with $g$ rapidly decreasing on $H$ and $f_1$ a smooth vector  for $G$.  Further, $(V_\lambda)^\infty$ endowed with the smooth topology is a Frechet representation for $G$. Whence, $(V_\lambda)^\infty$ is a Frechet representation for $H$. Therefore, since $f_1$ is a $G-$smooth vector we conclude,  $(\pi_\lambda)_{\vert_H} (g) (f_1)$ is an element of $(V_\lambda)^\infty.$ Thus, every $H-$smooth vector in $V_\lambda[V_\nu^H]$ is $G-$smooth.

\smallskip

Our hypothesis is that $res_H(\pi_\lambda)$ is an $H$-admissible representation. Thus, there exists a subset $Spec(res_H(\pi_\lambda))$ of the set of Harish-Chandra parameters for $H$ so that we have the Hilbert sum $V_\lambda^G =\oplus_{\nu \in Spec(res_H(\pi_\lambda))} V_\lambda[V_\nu^H]$ and $V_\lambda[V_\nu^H] \not= \{0\}$ if and only if $\nu \in Spec(res_H(\pi_\lambda))$. \\ Next, we write for  $z \in G, u \in (V_\lambda^G)^\infty $  $$\pi_\lambda (z) u= \sum_{\nu \in Spec(res_H(\pi_\lambda))} P_{\lambda, \nu}(\pi_\lambda(z)u) \eqno{(a)}.$$
We claim the convergence of the   series above is absolutely in both, $L^2-$to\-pol\-o\-gy and  the to\-pol\-o\-gy  for $(V_\lambda)^\infty$. Moreover, the series converges  in the topology of uniform convergence on  compact sets for functions as well for any derivatives.

Our hypothesis shows that $L^2-$convergence is obvious. Since $\pi_\lambda(z)u$ is a smooth vector, we have that for every $\nu$, the vector $P_{\lambda, \nu}(\pi_\lambda(z)u)$ is $H-$smooth, the previous claim shows $P_{\lambda, \nu}(\pi_\lambda(z)u)$ is $ G-$smooth. We recall a result of Harish-Chandra  \cite[Lemma 5]{HC2} which asserts: the Fourier series of a smooth vector converges absolutely in smooth topology. Therefore, the Fourier series of $\pi_\lambda (z)$ as well as the Fourier series for $P_{\lambda, \nu}(\pi_\lambda(z)u)$ converges absolutely in smooth topology. Since, $H-$admissible implies $L-$admissible, and that the subspace of $L-$finite vectors is equal to the subspace of $K-$finite vectors,  we have that the series  $P_{\lambda, \nu}(\pi_\lambda(z)u)$ is a rearrangement of a subseries of  the Fourier series for $\pi_\lambda(z)u$. Thus, the series $\sum_\nu P_{\lambda, \nu}(\pi_\lambda(z)u)$  converges absolutely in smooth topology. The third affirmation follows from \cite{At}.  \\ We recall, in \cite{HC2},  it is shown the smooth vectors in $\,H^2(G,\tau)$ are tempered functions and convergence in smooth topology implies convergence in the space of tempered functions.    Therefore, the series of functions
$$
h \mapsto (hy v, xw)_{L^2(G)}   = \sum_{\nu \in \,\, Spec(res_H(\pi_\lambda))} (h \mapsto(P_{\lambda, \nu}(hyv), xw)_{L^2(G)}) \eqno{(b)}.  $$
converges in the topology for the space of tempered functions. The equality follows from the  series (a) applied to $z=hyv$ and the continuity of the inner product on each variable.    Applying the equality $(P_{\lambda, \nu}(hyv), xw)_{L^2(G)}=(P_{\lambda, \nu}(hyv), P_{\lambda,\nu}xw)_{L^2(G)}$ in  (b), and, applying to the resulting series the character of $(\pi_\mu^H)^\star$,   we obtain
\begin{multline*} \lefteqn{\Theta_{ (\pi_\mu^H)^\star}( h \mapsto (hy v, xw))
        } \\  =\sum_{\nu \in Spec(res_H(\pi_\lambda))}  \Theta_{(\pi_\mu^H)^\star}(h \mapsto ( hP_{\lambda, \nu} (yv),P_{\lambda, \nu}(xw))_{L^2(G)}). \end{multline*}
Next, the function $h \mapsto ( hP_{\lambda, \nu} (yv),P_{\lambda, \nu}(xw))_{L^2(G)}$ is a matrix coefficient for $V_\lambda[V_\nu^H],$  besides, $P_{\lambda, \nu} (yv),P_{\lambda, \nu}(xw) $ are smooth vectors for $H$,    hence, the orthogonality relations as written in \cite[Lemma 84]{HC2}  gives us
\begin{multline*} \Theta_{(\pi_\mu^H)^\star}(h \mapsto ( hP_{\lambda, \nu} (yv),P_{\lambda, \nu}(xw))_{L^2(G)})\\ =\left\{ \begin{array}{ll} 0 & \mbox{ $\nu \not= \mu$} \\ \frac{1}{d_\mu} (P_{\lambda, \mu} (yv), P_{\lambda, \mu} (xw))_{L^2(G)} & \mbox{ $\nu=\mu$ .} \end{array} \right. \end{multline*}

Thus,

\begin{equation*}
\begin{split}
  \Theta_{ (\pi_\mu^H)^\star}( h \mapsto (hy v, xw))
           &=\frac{1}{d_\mu} (P_{\lambda, \mu} (yv), P_{\lambda, \mu} (xw))_{L^2(G)} \\
          &=\frac{1}{d_\mu} (P_{\lambda, \mu} (yv), xw)_{L^2(G)}.
  \end{split}
 \end{equation*}
After we apply the above equality to $v=f_j, w=f_i$ and add up,  we  conclude a   proof of Proposition~\ref{prop:klmchar}.
\end{proof}

\section{Criteria for discretely decomposable restriction}
As in previous sections,  we keep the hypothesis and notation of Section 2. The objects are $G, K, (\tau, W),  \,H^2(G,\tau), H, L.$    We recall the orthogonal projector
$P_\lambda $ onto $ \,H^2(G,\tau)$,  is given by a smooth matrix kernel $ K_\lambda(y,x)=K_\lambda(x^{-1}y,e)=\Phi_0(x^{-1}y)$ (cf. Appendix~\ref{subsec:A6}).
Here, $\Phi_0 $ is the spherical function associated to the lowest $K-$type $(\tau,W)$ of $\pi_\lambda^G$. Harish-Chandra showed that $\Phi_0$, and hence, $tr (\Phi_0) $ are tempered functions for the definition of Harish-Chandra, \cite{HC2}\cite[8.5.1]{Wa1}.  In \cite{HHO}, we find a proof that the tempered functions on $G$ restricted to $H$ are tempered functions. A tempered function is called a {\it cusp form} if the integral along the unipotent radical of any proper parabolic subgroup  of any left translate of the function is equal to zero \cite[7.2.2]{Wa1}. Let $r_n : H^2(G,\tau )  \rightarrow L^2(H\times_{\tau_n }  (\p/\p')^{(n)} \otimes W)$ be as in Example~\ref{examp:rn}. The notation $r_n(\Phi_0^\star )$ means the family of functions $r_n(K_\lambda(\cdot ,e)^\star w)=r_n( \Phi_0^\star(\cdot)w), w \in W.$ The previous considerations yield the family $r_n(\Phi_0^\star)$ consists of tempered functions.   The purpose of this section is to show:
\begin{thm}  \label{prop:cuspcrit}   Let $\pi_\lambda$ be a discrete series for $G$, let $\Phi_0$ be its lowest $K-$type spherical function.  Then,      $r_n(\Phi_0^\star) $  is a cusp form on $H$,   for every $n=0,1,\dots$, if and only if $\pi_\lambda^G$ restricted to $H$ is discretely decomposable. In turn, this is equivalent to: for each $y \in G$, the restriction   of $K_\lambda (\cdot ,y )$ to $H$ is a cusp form.
\end{thm}
\begin{rmk}  T.  Kobayashi \cite[Theorem 2.8]{KO2} has shown that for a symmetric pair $(G,H)$, $\pi_\lambda $ restricted to $H$ is algebraically discretely decomposable if and only if $res_H(\pi_\lambda )$ is $H-$admissible. Whence, coupling the previously quoted result of Kobayashi, with Theorem~\ref{prop:cuspcrit} and Proposition~\ref{prop:ddzhfinite} we may state: for a symmetric pair $(G,H),$ the restriction  of $\pi_\lambda$ to $H$  is admissible if and only if for every $n=0,1,\dots $ $r_n(\Phi_0^\star)$ is a cusp form   if and only if \\ $\Phi_0^\star$ is left $\z(U(\h))-$finite.\\ For a symmetric pair $(G,H),$ another criterion for $H-$admissibility has been obtained by \cite{HHO}. For this, they write $H=K_0 \times H_1$, with $K_0$ a compact subgroup and $H_1$ a noncompact semisimple subgroup. Let $H^{\sigma  \theta} $ be the dual subgroup. Then, $H^{\sigma  \theta}=K_0 \times H_2$. Let  $M_i$ denote the centralizer in $L\cap H_1 =L\cap H_2$  of respective Cartan subspaces.   Harris-He-Olafsson show that if $M_1 M_2 =L\cap H_1$, then,  the representation $\pi_\lambda$ restricted to $H$ is admissible, and    $\dim Hom_H (\pi_\mu^H, res_H(\pi_\lambda) )$ is computed via a formula that involves $r_n$, the Harish-Chandra character for $\pi_\mu^H$,   the lowest $L-$type for $\pi_\mu^H$ and the limit of a sequence. \\
	In \cite{DV}, for any pair $(G,H)$ and $\pi_\lambda$ that satisfies Condition C,   it is shown: $\pi_\lambda$ is a $H-$admissible representation;  a "Blattner-Kostant" type formula for  $\dim Hom_H (\pi_\mu^H, res_H(\pi_\lambda) )$. For a symmetric pair $(G,H)$,   Condition C is equivalent to $H-$admissibility.
\end{rmk}
In order to show Theorem~\ref{prop:cuspcrit} we first show
\begin{prop} \label{prop:ddzhfinite} We let  $G, H, \pi_\lambda, K_\lambda, \Phi_0 $ be as in Theorem~\ref{prop:cuspcrit}. Then, $\pi_\lambda$  restricted to $H$ is a discretely decomposable representation for $H$ if and only if the   function $y \mapsto K_\lambda(y,e)^\star =\check{\Phi}_0 (y)$    is left $\mathfrak z (U (\mathfrak h))-$finite.
\end{prop}
\begin{proof} For the direct implication, we proceed as follows.  Our hypothesis is $\pi_\lambda$ is discretely decomposable; this  allows us to write $V_\lambda$ as the Hilbert sum of the $H-$isotypic components; hence,   there exists a family $(P_i)_{ i \in \Z_{\geq 0}}$ of orthogonal projectors on $V_\lambda$ which are $H-$equivariant so that we have the orthogonal direct sum decomposition $V_\lambda =\oplus_i P_i(V_\lambda),$ and for every $i$,  $P_i(V_\lambda)$ is equal to the isotypic component of an irreducible $H-$module.  Next, we fix $w \in W$, we recall   the function
	$ y \mapsto K_\lambda(y,e)^\star (w)=:k_w(y)$  is a $K-$fi\-ni\-te  element of $\,H^2(G,\tau)$  and $k_w$ belongs to $\,H^2(G,\tau)[W]\equiv W .$  After we decompose $\,H^2(G,\tau)[W]$ as  a sum of irreducible $L-$submodules, we write $k_w=f_1 +\dots +f_s$,  where $f_j$ is so that the linear subspace spanned by  $\pi_\lambda(L)f_j$ is an irreducible $L-$submodule of $\,H^2(G,\tau)[W].$ To continue, we set $f_1:=f_j.$ A result of Harish-Chandra, \cite[Lemma 70]{HC2}, states  that an irreducible representation of $L$ is the $L-$type of at most finitely many discrete series representations for $H.$ Thus,  the representation of  $L$ in the  subspace spanned by $\pi(L)f_1$ is an $L-$type of at most finitely many discrete series representations for $H.$ Therefore, $P_i(f_1)=0 $ for all but finitely many indices $i$. Let us say $P_i(f_1)\not=0$ for $i=1, \dots, N.$ Since $f_1$ is a $K-$finite vector in $V_\lambda$, we have that $f_1$ is a smooth vector for $\pi_\lambda,$ therefore, $P_i(f_1)$ is a $H-$smooth vectors in $P_i(V_\lambda).$  Owing to $P_i(V_\lambda)$ is an isotypic representation, we have that $\z(U(\h))$ applied to $P_i(f_1)$ is contained in the one dimensional vector subspace spanned by $P_i(f_1)$. Hence, $f_1$ is a finite sum of $\z(U(\h)-$finite vectors. Thus, $k_w$ is a  $\z(U(\h))-$finite vector.  $W$ is a finite dimensional vector space,  let us conclude that the map $  K_\lambda( \cdot,e)^\star=\Phi_0 (\cdot)^\star$ is left $\z(U(\h))-$finite.

	\smallskip
	
	For the converse statement, owing to our hypothesis, for each $w \in W $ we have that  $k_w$  is  $\mathfrak z (U (\mathfrak h))-$finite element of $\,H^2(G,\tau)_{K-fin}.$  A result of Harish-Chandra,  \cite[Corollary  3.4.7 and Theorem 4.2.1]{Wa1}, asserts that a $U(\h)-$finitely generated, $\z(U(\h))-$finite, $(\h,L)-$module has a finite composition series, whence, we conclude that the representation of $U(\h)$ in  $U(\h)k_w$ has a finite composition series.  Thus,  $\,H^2(G,\tau)_{K-fin}$ contains an irreducible subrepresentation for $U(\h).$ Hence, \cite[Lemma 1.5]{K3} yields that  $H^2(G,\tau )_{K-fin}$ is infinitesimally discretely decomposable as $\h-$mod\-ule. Finally, since $\pi_\lambda$ is unitary,   in \cite[Theorem 4.2.6]{K4},   we find a proof that an   algebraically (infinitesimally) discretely decomposable unitary representation is Hilbert discrete decomposable; hence,     $\pi_\lambda$ is  discretely decomposable.
\end{proof}

Now, we are ready to show  Theorem~\ref{prop:cuspcrit}.  \begin{proof}[Proof of Theorem~\ref{prop:cuspcrit}]
	For the direct implication, the hypothesis is that $\pi_\lambda$ restricted to $H$ is discretely decomposable. Thus, Proposition~\ref{prop:ddzhfinite}	yields \\ $K_\lambda(\cdot ,e)^\star w $ is $\z(U(\h))-$finite. Since in \cite{OV}, it is shown  $r_n$ is a continuous intertwining map for $H$, and $K_\lambda(\cdot, e)^\star w$ is a tempered function, a result of \cite{HHO} previously quoted, lets us conclude: $r_n (K_\lambda(\cdot, e)^\star w) $ is a tempered, $\z (U(\h))-$finite function on $H.$ A result of  Harish-Chandra, \cite{HC2}\cite[7.2.2]{Wa1}, implies $r_n (K_\lambda(\cdot, e)^\star w) $ is a cusp form. For the converse statement,
	results of Harish-Chandra asserts $L^2(H\times_{\tau_n} (\p/\p')^{(n)} \otimes W)_{disc}$ is a finite sum of discrete series representations for $H$ and the space of cusp forms in $L^2(H\times_{\tau_n} (\p/\p')^{(n)} \otimes W)$ is contained in $L^2(H\times_{\tau_n} (\p/\p')^{(n)} \otimes W)_{disc}$. Therefore, owing to our hypothesis,  for every $n$,  $r_n (K_\lambda(\cdot, e)^\star w) $ belongs to $L^2(H\times_{\tau_n} (\p/\p')^{(n)} \otimes W)_{disc}$. The $L^2-$continuity of $r_n$ yields $r_n(closure(\pi_\lambda (H)k_w))$ is contained in a finite sum of discrete series representations. Whence, $\oplus_n r_n$ maps continuously the closure of $\pi_\lambda (H)k_w$ into a discrete Hilbert sum of irreducible representations. Besides, the map $\oplus_n r_n$ is injective (the elements of $\,H^2(G,\tau)$ are real analytic functions).  Hence,   the closure of $\pi_\lambda (H)k_w$ is a discrete Hilbert sum of discrete series representations. We now proceed as in the direct proof Proposition~\ref{prop:ddzhfinite} and obtain $k_w$ is a left $\z (U(\h))-$finite function. Whence, Proposition~\ref{prop:ddzhfinite} lets us conclude $res_H(\pi_\lambda)$ is Hilbert discretely decomposable. 	The second equivalence follows from a simple computation.
\end{proof}	

\begin{rmk} A simple application of Theorem \ref{prop:cuspcrit} is that the tensor product representation in  $\,H^2(G,\tau)\boxtimes \,H^2(G,\tau)^\star$ restricted to diagonal of $G$ is never  discretely decomposable, because the lowest $K-$type trace spherical function for this particular tensor product is $\phi_0(x) \overline{\phi_0(y)},$ hence, restricted to $G$ it is not a cusp form.
\end{rmk}

 \section{Reproducing kernels and existence of discrete factors}\label{sec:repkediscfac}
  As usual $G,H,\,H^2(G,\tau)=V_\lambda, K_\lambda, H^2(H,\sigma)=V_\mu^H, K_\mu$ are as in Section~\ref{sec:prelim}. The purpose of this section is to begin an analysis of the relation between the matrix kernel $K_\lambda $ of a discrete series for $G$,  the matrix kernel $K_\mu$  of a discrete series for $H$ and the existence of a nonzero $H-$intertwining linear map from one representation  into  the other.  To begin with, since $K_\lambda$ is given by the spherical function associated to the lowest $K-$type of $\pi_\lambda^G$, the restriction of $y \mapsto tr (K_\lambda(y,x))=:k_\lambda (y,x)$ to $H$ is a square integrable function \cite{OV}. Let $k_\lambda :=tr K_\lambda, k_\mu:=tr K_\mu.$      We recall \cite{Va2},
  \begin{prop}\label{prop:klambdakmuate}
  	$ (k_\lambda(\cdot,e), k_\mu( \cdot,e))_{L^2(H)}$  nonzero implies  $V_\mu^H$ is a discrete factor for $res_H(\pi_\lambda^G).$
  \end{prop}
   The Proposition is an immediate consequence of the following
  \begin{fact}\label{prop:discretefactor} \cite{Va2} Let $G,H,\pi_\lambda$ be as usual. Let $(\rho , V_1)$ an irreducible square integrable for $H$. We assume there exists smooth vector $w_1,w_2$ in $ V_\lambda$ and a $L-$finite vector $z$ in $V_1$ so that $$\int_H (\rho(h) z,z)_{V_1} \overline{( \pi_\lambda (h)w_1, w_2)_{V_{\lambda}}}dh \not= 0. $$ Then, there exists a nonzero, continuous linear $H-$map from $V_1$ into $V_\lambda.$
  \end{fact}
  \begin{rmk} \label{rmk:contraejemplo} The converse to Proposition~\ref{prop:klambdakmuate}  is false, as it shows the following example. We consider   arbitrary $G,H$ and we assume $res_H(\pi_\lambda^G)$ is an $H-$admissible representation. Since $\,H^2(G,\tau)[W]$ consists of $L-$finite vectors, the subspace $\,H^2(G,\tau)[W]$ is contained in a finite sum of irreducible $H-$subrepresentations. Let $\pi_\mu^H $ be a representation for $H$ whose corresponding isotypic component is nonzero and it is  contained in the orthogonal to the subspace spanned by    the $H-$isotypic  components whose intersection with  $\,H^2(G,\tau)[W]$ is nonzero.  Since, for $w \in W$,  $K_\lambda(\cdot,e)^\star w \in \,H^2(G,\tau)[W]$,  we have that both $tr(K_\mu(\cdot,e))=k_\mu(\cdot,e)$ and $tr(K_\lambda (\cdot,e))=k_\lambda (\cdot,e)$ are linear combination of matrix coefficients of nonequivalent representations for $H,$ the orthogonality relations implies $ (k_\lambda(\cdot,e), k_\mu(\cdot,e))_{L^2(H)}= 0.$ \end{rmk}  A way to show a converse to Proposition~\ref{prop:klambdakmuate} is to consider the functions
  $$ \Upsilon_{\mu,\lambda} (y,x)= \int_H  k_\lambda (hy,x) K_\mu(h,e)^\star  \, dh. $$   $$ \upsilon_{\lambda, \mu} (y,x):= \int_H k_\mu(h_1,e) K_\lambda (h_1y , x)^\star  dh_1.$$ Both integrals converge absolutely for every $x,y \in G$ owing to in \cite{OV} it is shown both integrands are product of square integrable functions. Moreover, the  functions $\Upsilon,$ $\upsilon$ are  continuous  because each integral can be rewritten as a convolution product  and we recall the fact: convolution of two square integrable functions is defined and yields a continuous function. \\
 Under the supposition  the  representation $  \pi_\mu^H$ is integrable it readily follows the smoothness of both functions as well as that each function is a Carleman kernel. In fact,  we notice $z\mapsto K_\lambda(zy,x))$ belongs to either  the Frechet space $(L^2(G)\otimes Hom_\C(W,W))^\infty$ or to   the Hilbert space $L^2(G)\otimes Hom_\C(W,W)$ and  $k_\mu (\cdot,e) \in L^1(H)$. Next, we  write $\upsilon (y,x)=res_H(L)_{(k_\mu(\cdot,e))} (z\mapsto K_\lambda(zy,x))$   Therefore,  both functions $\upsilon(\cdot,x), \upsilon (y,\cdot) $ are smooth and square integrable. For the function $\Upsilon$, we write $\Upsilon (y,x)= res_H(L)_{(K_\mu(\cdot,e)^\star)}(z\mapsto k_\lambda (zy,x))$ and proceed as for $\upsilon.$
      Among properties of the functions are: \\ 0)  $tr(\Upsilon (e,e))= (k_\lambda (\cdot,e), k_\mu(\cdot,e))_{L^2(H)},$  $tr(\Upsilon (y,x))=\overline{tr (\upsilon(y,x))}$. \\ 1) $\Upsilon (s_1y
  k,s_2xk)=  \sigma(s_1) \Upsilon_{\mu,\lambda}(y,x) \sigma (s_2^{-1}) , \text{for} \, x \in G, s_j \in L, k \in K.$\\
  2) The function $\Upsilon_{\mu,\lambda}$ is real analytic. \\  This is because
  the distribution on $G$ defined by either the function $y\mapsto \Upsilon_{\mu,\lambda} (y,x) $ or $x \mapsto \Upsilon_{\mu,\lambda} (y,x)$ is a real analytic $L-$spherical   function. Indeed, both distributions are eigenfunctions of the elliptic differential operator $R_{\Omega_G} +2 R_{\Omega_K}$,   the regularity Theorem
  leads us to the real analyticity.\\
  3) $\Upsilon_{\mu,\lambda} (y,x)= (\Upsilon_{\mu,\lambda} (x,y))^\star.$\\
  4) Remark~\ref{rmk:contraejemplo} shows that some times $\Upsilon(e,e)=0$.  \\
  5) $\upsilon(syk_1,sxk_2)=\tau(k_2^{-1})\upsilon(y,x)\tau(k_1), k_1,k_2 \in K, s \in L$.\\
  The main result of this subsection is:
  \begin{thm} \label{prop:twofunctions} Let $G,H, (\pi_\lambda, \,H^2(G,\tau))$, $ (\pi_\mu^H, V_\mu^H =H^2(H, \sigma ))$ as in Section~\ref{sec:prelim}. The following  four  conditions are equivalent
  	
  	a) The isotypic component $\,H^2(G,\tau)[V_\mu^H]$ is not zero.
  	
  	b) The function $\Upsilon_{\mu,\lambda} $ is nonzero.
  	
  	c) The function  $tr(\Upsilon_{\mu, \lambda}) $ is nonzero.
  	
  	d) There exists $D \in U(\g \times \g) $ so that $[L_D (\Upsilon)] (e,e) \not= 0. $
  	
  	e) The function $\upsilon_{\lambda, \mu}$ is nonzero.
  \end{thm}
  \begin{proof}[Proof of Theorem~\ref{prop:twofunctions}] We show  $a)\Rightarrow b) \Leftrightarrow c) \Rightarrow e) \Rightarrow a) $ and $b) \Leftrightarrow d).$ Some of the implications are obvious.
  	{\it a)} implies {\it b)}.  By hypothesis, there exists  a
  	nonzero intertwining map  $T : H^2(H,\sigma) \rightarrow \,H^2(G,\tau)$. Owing to Schur's Lemma $T$ is an injective map. Let $K_T : H\times G \rightarrow Hom_\mathbb C (Z,W)$ be the kernel that represents $T$. We know, for $z \in Z,$ the function $h \mapsto K_\mu(h,e)^\star z$ belongs to $H^2(H,\sigma)$ and it is nonzero. Thus, $T(K_\mu(\cdot,e)^\star z)$ is a nonzero function  for each $z \in Z.$ Since $k_\lambda$ represents $P_\lambda$ we obtain
  	$$\int_G  \int_H k_\lambda (y,x) K_T(h,y) K_\mu(h,e)^\star z\, dh dy \not= 0 \, \text{\, for\,some\, }\, x \in G. $$
  	Now $K_T(h,y)=K_T(e,h^{-1}y).$ Thus, after we replace $y$ by $hy$, we recall Haar measure is left invariant and $k_\lambda$ is a complex valued function,  the left hand side becomes
  	$$\int_G  K_T(e,y)\int_H  k_\lambda (hy,x) K_\mu(h,e)^\star z \, dh dy . $$  Thus, $\Upsilon $ is not the zero function. We now show $b) \Leftrightarrow c).$  The equality 1) forces the linear span of the image of $\Upsilon_{\mu,\lambda}$ is an $L\times L$ invariant subspace of $End_\C (Z).$ Thus, owing to $End_\C(Z)$ is an $L\times L-$module irreducible, we have there exists $(y,x)$ so that $tr(\Upsilon_{\mu,\lambda}(y,x))\not= 0$ unless  $\Upsilon_{\mu,\lambda}$ is equal to the zero map.
  	$c) \Rightarrow e)$ follows from  $tr(\upsilon_{\lambda, \mu})=\overline{tr(\Upsilon_{\mu,\lambda})}.$
  	To show $e) \Rightarrow a)$. We notice that the hypothesis implies we may apply Fact~\ref{prop:discretefactor}. Therefore the isotypic component is nonzero. The two functions are real analytic, thus, we conclude the proof. \end{proof}
  \subsection{Examples of discrete factors in $res_H(\pi_\lambda)$}\label{subsec:somediscrfactors}
  As before, we  fix  $G,H, \\ K,  L $ and the representation  $(\pi_\lambda, H^2(G, \tau))$.  A consequence of Proposition~\ref{prop:klambdakmuate}  is the following result   shown in \cite{Va},
  \begin{fact}\label{fact:fact:A}   We suppose   $(\sigma, Z)$ is an $L-$subrepresentation of the lowest $K-$type $(\tau, W)$. We further assume there exists a discrete series representation  $H^2(H,\sigma)$ with lowest $L-$type $(\sigma, Z).$ Then, there exists a nonzero intertwining map from $H^2(H,\sigma)$ into $\,H^2(G,\tau)$. \
  \end{fact}
   The proof for this fact  is based on an explicit integral formulae for $k_\mu, k_\lambda$ obtained by Flensted-Jensen, these formulae  let us to apply Proposition~\ref{prop:discretefactor}.
  We present one application   of   fact~\ref{fact:fact:A} to  the analysis of the tensor product of two representations.  For other applications  (cf. \cite{Va2}).
  \begin{examp} Notation is as in Section~\ref{sec:prelim}.     We  produce some irreducible subrepresentation of
  	the restriction to the diagonal  subgroup $H=G_0$  of $G:=G_0 \times G_0$ ($G_0$ a semisimple Lie group) for  a tensor product
  	$\pi_\lambda \boxtimes \pi_{\lambda_1}. $    We choose both  Harish-Chandra   parameters to be dominant with respect to $\Psi$ and   so that
  	their sum is far away from the kernel of any noncompact simple  root. It readily follows that  $\pi_{\lambda +\lambda_1 +2\rho_n-\rho_c}^{K_0}$ is a subrepresentation of  the lowest $K_0\times K_0-$type for $\pi_\lambda \boxtimes \pi_{\lambda_1}. $ Due to our choice, the parameter   $\lambda +\lambda_1 +2\rho_n-\rho_c$ is far  from the noncompact walls for the Weyl chamber for $\Psi.$ It follows (cf. \cite{DHV}) that $\pi_{\lambda +\lambda_1 +2\rho_n-\rho_c}^{K_0}$  is the lowest $K_0-$type of the  discrete series $\pi_{\lambda +\lambda_1 +\rho_n-\rho_c}^{G_0}$.  Therefore, Fact~\ref{fact:fact:A} implies $res_{G_0}(\pi_\lambda \boxtimes \pi_{\lambda_1})$ contains the irreducible representation $\pi_{\lambda +\lambda_1 +\rho_n-\rho_c}^{G_0}.$ \end{examp}

  \section{Proof of Proposition~\ref{prop:propertieskt} and Theorem~\ref{prop:symbrea}}

  \begin{proof}[Proof of Proposition~\ref{prop:propertieskt}] A straightforward computation based on the invariance of Haar measure gives $f)$. To show $b)$, we notice $K_T(s^{-1}h,e)^\star w = K_T(h,s)^\star w= (\tau(s^{-1})K_S(h,e))^\star w = K_T(h,e)^\star \tau(s)w.$ Thus, dimension of   the span of $\{L_s (h\mapsto K_T(h,e)^\star w) : s \in L \}$ is less or equal than $dim W$.
	Next, we show  $a)$.
	Let $g \in L^2(H\times_\nu E)$ arbitrary,   we have  \begin{align*}
	(g(\cdot), K_T(\cdotp, y)^\star (w))_{L^2(H \times_\nu E)}  & = (Tg(y),w)_W \\ \nonumber & = (Tg(\cdot) ,K_\lambda (\cdot,y)^\star (w))_{L^2(G \times_\tau W)}\\  & = (g(\cdot), T^\star (t \mapsto K_\lambda(t,y)^\star (w))(\cdot))_{L^2(H \times_\nu E)}.
	\end{align*}  The first equality is justified by \ref{eq:tisintegral},    the second by   $P_\lambda (Tg)=Tg$ and (\ref{eq:Klambda}), the  third by definition of adjoint linear map. Since the functions $v \mapsto  K_T(v, y)^\star (w), v \mapsto [T^\star (t \mapsto K_\lambda(t,y)^\star (w))](v)$ belong to $L^2(H\times_\nu E)$ and $g$ is arbitrary in $L^2(H\times_\nu E)$ we obtain    a).\\ For further use we denote the subspace of smooth vectors in a representation $V$ by $V^\infty$. \\
	Next, we verify the map $K_T$ is smooth and that $K_T(\cdot ,x)^\star w$ is a smooth vector. The smoothness of $K_T$ follows from  $(h,x) \mapsto K_T (h,x)^\star w $ is smooth for each $w \in W.$ We first notice $K_\lambda (y,x) = K_\lambda (e,y^{-1} x)$ and $x \mapsto K_\lambda (x,e)^\star w$ is a $K-$finite vector in
	$\,H^2(G,\tau)$. Thus, $y \mapsto K_\lambda (y,x)^\star w=L_x( y \mapsto K_\lambda (y,e)^\star w)$ is a smooth vector in $ \,H^2(G,\tau)$ and hence
	the map $ G \ni x \mapsto  K_\lambda (\cdot ,x)^\star w) \in  \,H^2(G,\tau)$ is smooth. Since $T^\star$ is
	a continuous linear $H-$map, we have that  $ K_T(\cdot ,x)^\star w= T^\star ( y\mapsto K_\lambda (y,x)^\star w)(\cdot) \in  L^2(H\times_\nu E)^\infty $,  hence $G \ni x \mapsto T^\star ( y\mapsto K_\lambda (y,x)^\star w) \in L^2(H\times_\nu E)^\infty$  is a smooth map. Next,  \cite[section 1.6]{Wa1}),   we endow the space of smooth vectors in a representation $(\pi ,V)$ with the topology that a sequence of smooth vectors $v_n$ converges to a smooth vector $v$ iff $\pi(D)v_n$ converges to $\pi(D)v$ in norm for every $D \in U(\h)$.   Poulsen \cite[Proposition 5.1]{Po} has shown that the space of smooth vectors $L^2(H\times_\nu E)^\infty$ in $L^2(H\times_\nu E)$ is the subspace of smooth functions $f$ so that $L_D (f)$ is square integrable for every $D \in U(\h).$ Further, Poulsen showed point  evaluation from $L^2(H\times_\nu E)^\infty$ into $E$ is a continuous linear map. Whence, the the following composition gives a smooth map  $ (h,x) \mapsto h^{-1}x \mapsto T^\star (v \mapsto K_\lambda (v, h^{-1}x)^\star w) (e).$ Now $a)$ and the equality $h_T(e,h^{-1}x)=h_T(h,x)$ concludes the proof of  $c)$.\\ $d)$ follows from   Sobolev's inequality     $ \Vert f(e)  \Vert_E \leq \sum_{1 \leq b \leq N} \Vert L_{X_b} (f) \Vert_{L^2(H)}$, as shown by Poulsen,  \cite[Lemma 5.1]{Po},  applied to $f(h):=K_T(h,x)^\star w$ at $h=e$. In fact, owing to c),   $K_T(\cdot,x)^\star w$ is an smooth vector and we have $L_{X_b}^{(1)} K_T(h, x)^\star w=L_{X_b} (T^\star (K_\lambda(\cdot,x)^\star w)(h) =T^\star (L_{X_b}^{(1)} K_\lambda (\cdot ,x)^\star w)(h)$. Thus,  $ \Vert L_{X_b}^{(1)} K_T(\cdot, x)^\star w \Vert_{L^2(H)} \leq \Vert T^\star \Vert \,\Vert L_{X_b} K_\lambda (\cdot ,x)^\star w\Vert_{L^2(G)} .$
	
The equality $K_\lambda (y,x)=K_\lambda (x^{-1}y,e)$ together with the left invariance of Haar measure yields $\Vert L_{X_b} K_\lambda (\cdot ,x)^\star w\Vert_{L^2(G)}=\Vert L_{Ad(x^{-1})X_b } K_\lambda (\cdot ,e)^\star w\Vert_{L^2(G)}.$ Next, we fix a basis $\{Y_a \}_{1\leq a \leq M}$ for the subspace of elements of $U(\g)$ of degree less or equal to $\dim \h.$   We  write $Ad(x^{-1} )X_b= \sum_a \phi_{a,b}(x)Y_a$.
	  Let $C $ denote an upper bound for the numbers $ \Vert L_{Y_a} K_\lambda (\cdot ,e)^\star \Vert_{L^2(G)} $ for $1 \leq a \leq M.$  Then, Sobolev's inequality applied to $K_T(\cdot ,x)^\star w$ at $h=e$, together with the previous inequalities gives $ \Vert K_T(e,x)^\star w \Vert_W   \leq \Vert T^\star \Vert \sum_{1\leq b \leq N} \Vert L_{X_b} K_\lambda (\cdot ,x)^\star w\Vert_{L^2(G) } \\ = \Vert T^\star \Vert \sum_b   \Vert L_{Ad(x^{-1})X_b} K_\lambda (\cdot ,e)^\star w\Vert_{L^2(G)} \leq C \Vert T^\star \Vert   \Vert w \Vert_W \sum_{a,b} \vert  \phi_{a,b}(x)\vert , $  and \\  $\Vert K_T(e,x)^\star \Vert_{Hom(W,Z)}    =\sup_{\vert w \Vert \leq 1}\{\Vert K_T(e,x)^\star w\Vert_W\}   \leq  C \Vert T^\star \Vert    \sum_{a,b} \vert  \phi_{a,b}(x)\vert . $ Thus, $d)$ follows.\\   Now, we verify $e)$. That is,  the function   $x \mapsto \Vert K_T(\cdot,x)^\star \Vert_{L^2(H)} $ is  bounded, for this,   we verify the  inequality  \\ \phantom{xxxxxxxxx} $  \Vert K_T(\cdot,x)^\star \Vert_{L^2(H)} <<  \Vert K_\lambda (\cdot,e) \Vert_{L^2(G)} \Vert\, \Vert T \Vert $.\\
	To begin with, for $x \in G, w \in W$,   we have\\ \phantom{xxxxxxx} $   \Vert   K_T(\cdot,x)^\star w \Vert_{L^2(H)} \leq  \Vert w\Vert \Vert K_\lambda  (\cdot,e) \Vert_{L^2(G)} \Vert\, \Vert T \Vert $. Indeed,
	\begin{alignat*}{2}
	\int_H \Vert   K_T(h,x)^\star w\Vert_Z^2 dh    &  =\int_H ( K_T(h,x)^\star w,  K_T(h,x)^\star w)_E dh \\
	& \, = (T( K_T(\cdot,x)^\star w)(x),w)_W \\  & =\int_G (T( K_T(\cdot,x)^\star w)(y), K_\lambda(y,x)^\star w)_W dy \\
	& \leq \Vert T(  (K_T(\cdot , x)^\star w)\Vert_{L^2(G)} (\int_G \Vert K_\lambda (y,x)^\star w\Vert_W^2 dy )^{1/2} \\
	& \, \leq \Vert w \Vert \, \Vert K_\lambda(\cdot ,e) \Vert_{L^2(G)} \Vert T( K_T(\cdot,x)^\star w) \Vert_{L^2(G)} \\
	& \, \leq \Vert w \Vert \, \Vert K_\lambda (\cdot,e) \Vert_{L^2(G)} \Vert\, \Vert T \Vert \, \Vert  K_T(\cdot,x)^\star w \Vert_{L^2(H)}.
	\end{alignat*}
	Therefore,   after we simplify, we obtain
	$$ \int_H \Vert K_T(h,x)^\star w\Vert_Z^2 dh\leq \Vert w \Vert \, \Vert K_\lambda(\cdot,e) \Vert_{L^2(G)} \Vert\, \Vert T \Vert \, \Vert K_T(\cdot,x)^\star w \Vert_{L^2(H)}. $$
	Thus, $$ \Vert K_T(\cdot,x)^\star w \Vert_{L^2(H)} \leq \Vert w \Vert \, \Vert K_\lambda(\cdot,e) \Vert_{L^2(G \times_{\tau^\star \otimes \tau} Hom_\C (W,W))} \, \Vert T \Vert. $$
	To continue, we fix an orthonormal basis $\{z_j\}$ (resp. $\{w_i\}$) for $E$ (resp. for $W$). Then, there exists a constant $C_1=C_1(E,W)$,  so that for each linear  map  $ R : W\rightarrow E$ we have,  $\Vert R \Vert_{Hom(W,E)} \leq C_1 \sqrt{ \sum_{i,j} \vert(R w_i,z_j)_E \vert^2}.$ Therefore,
	\begin{multline*}  \big (\Vert K_T(\cdot ,x)^\star \Vert_{L^2(H\times_{\tau^\star \otimes \nu} Hom_\C (W,E)}\big )^2 \\ \leq C_1^2 \int_H \sum_{i,j}\vert (K_T(h,x)^\star w_i, z_j)_E \vert^2 dh  \leq C_1^2 \sum_{i} \dim E \int_H  \Vert K_T(h,x)^\star w_i\Vert_E^2 dh \\ \leq C_1^2 \dim E \dim W  \Vert K_\lambda(\cdot, e)^\star \Vert_{L^2(G)}^2 \Vert T \Vert^2 .   \end{multline*} Whence, for every $x \in G$, we obtain  \begin{multline*} \Vert K_T(\cdot,x)^\star \Vert_{L^2(H\times_{\tau^\star \otimes \nu} Hom_\C (W,E))}\\ \leq C_1 \sqrt{\dim E \dim W}  \Vert K_\lambda (\cdot,e) \Vert_{L^2(G\times_{\cdot} Hom(W,W))} \Vert\, \Vert T \Vert. \end{multline*} This concludes the proof of Proposition~\ref{prop:propertieskt}.
\end{proof}

\begin{proof}[Proof of Theorem~\ref{prop:symbrea} ]
	In order to show $a)$, we recall that Schur's Lemma implies that for a closed $H-$irreducible subspace $N$ of $\,H^2(G,\tau)$ either $S (N)=\{0\}$ or $S(N)$ is a  closed $H-$irreducible subspace. Thus, our hypothesis forces the image of $S$ is equal to a sum of $H-$irreducible  subspaces of  $L^2(H\times_\nu E)$. Plancherel's Theorem of Harish-Chandra shows that closure of the sum of the $H-$irreducible subspaces in $L^2(H\times_\nu E)$, is equal to a finite sum of irreducible subspaces; this follows from Frobenius reciprocity and \cite[Lemma  72]{HC2}. From now on, $L^2(H\times_\nu E)_{disc}$ denotes the sum of $H-$invariant irreducible subspaces in $L^2(H\times_\nu E)$. On the smooth vectors of a $H-$irreducible subspace $N$ the Casimir operator acts by a constant since smooth vectors are smooth functions, the Casimir operator acts by the same constant on the whole of $N.$ In \cite{At}, we find a proof that a $L^2-$eigenspace of an elliptic operator on a fiber bundle is a reproducing kernel subspace. Finally, the Casimir operator of $\h$ acts as an elliptic operator on $L^2(H\times_\nu E)$. Thus, we conclude that  $L^2(H\times_\nu E)_{disc}$ is a finite sum of reproducing kernel subspaces;  hence,  $L^2(H\times_\nu E)_{disc}$ is a reproducing kernel subspace.  Whence, the image of $S$ is contained in a reproducing kernel subspace, which lets us conclude $S$ is a kernel map. It readily follows  kernel $K_S$ of $S$ is equal to  $K_{S^\star}^\star. $
	
	Subsequently, we show $b)$. The hypothesis  $S$ is a continuous $H-$map and Proposition~\ref{prop:propertieskt} yields that $S^\star$ is an integral map. Thus, at least formally, we think   $S$ as the adjoint of an integral map. For this, we  formally   define $$ S_0(f)(h):= \int_G K_{S^\star}(h,x)^\star f(x) dx$$ and we consider  the subspace $$ \mathcal D_{K_{S^\star}^\star} := \{ f \in L^2(G \times_\tau W) : S_0(f) \in L^2(H \times_\nu  E) \}.$$
	It readily follows that $S_0$ restricted to $\mathcal D_{K_{S^\star}^\star}$ is an integral operator with kernel $K_{S_0}(x,h)=K_{S^\star}(h,x)^\star. $ Besides, $S$ is an integral map when restricted to $\,H^2(G,\tau)\cap \mathcal D_{K_{S^\star}^\star}$. To follow, we  construct a  subspace $\mathcal D$ of $\mathcal D_{K_{S^\star}^\star}  ,$  $$\mathcal D:=\{ f \in L^2(G\times_\tau W) : \int_G \Vert K_{S^\star}(\cdot,x)^\star  \Vert_{L^2(H\times_\nu Hom(W,E))} \Vert f(x)\Vert dx <\infty \}  $$
	For this,  we show that for $f \in \mathcal D$, the  integral that formally  defines the function $h \mapsto S_0 (h)  $ is absolutely convergent almost everywhere in $h$  and the resulting function belongs to $L^2(H\times_\nu E).$  We apply the integral version of Minkowski's inequality for $p=2$ and we obtain
	\begin{multline*}
	(\int_H   [\int_G \Vert K_{S^\star}(h,x)^\star f(x)\Vert_E dx ]^2 dh)^{1/2} \\ \leq \int_G   [\int_H \Vert K_{S^\star}(h,x)^\star f(x)\Vert_E^2  dh]^{1/2} dx \\ \leq \int_G   [\int_H \Vert K_{S^\star}(h,x)^\star \Vert_{Hom(W,E)}^2 dh]^{1/2} \Vert f(x)\Vert_W  dx.
	\end{multline*}
	The right hand side of the previous inequality is a finite number because $f$ belongs to $\mathcal D.$ Hence,  $\int_G K_{S^\star}(h,x)^\star f(x) dx $ is absolutely convergent almost everywhere in $h$ and the resulting function belongs to $ L^2(H\times_\nu E).$

	As a consequence, we obtain that for $g \in L^2(H\times_\nu E)$ and  $ f \in \mathcal D$
	the following two iterated integrals are absolutely convergent   $$ \int_G \int_H( K_{S^\star}(h,x) g(h), f(x))_W dh dx = \int_H \int_G( K_{S^\star}(h,x) g(h), f(x))_W dx dh .$$
	We are ready to conclude the proof of Theorem~\ref{prop:symbrea} b).
	Whenever $\pi_\lambda$ is an  integrable discrete series representation, it follows from the work of Harish-Chandra, \cite[Lemma 76]{HC1}, that  any $K-$finite vector $f$ in $\,H^2(G,\tau)$ is integrable with respect to Haar measure on $G$. We claim that any smooth vector in $\,H^2(G,\tau)$ is an integrable function. For this we recall the space of rapidly decreasing functions $\mathscr S(G)$ on $G$, defined by \cite[page 230]{Wa1}.  Owing to \cite[ Lemma 2.A.2.4]{Wa1},  any  rapidly decreasing function is integrable. Next, in \cite[Theorem 11.8.2]{Wa2},   it is shown that the subspace $\,H^2(G,\tau)^\infty$ of smooth vectors in $\,H^2(G,\tau)$ is an algebraically irreducible module over $\mathscr S(G)$. Therefore, any smooth vector in $\,H^2(G,\tau)$ is equal to the convolution of a rapidly decreasing function on $G$ times a $K -$finite vectors. Hence, any smooth vector is convolution of two integrable functions on G. Classical harmonic analysis yields any smooth vector in $\,H^2(G,\tau)$ is an integrable function.     Therefore, for  a smooth vector  $f$ in $\,H^2(G,\tau)$,  Proposition~\ref{prop:propertieskt} $e)$, forces  the integral   $$ \int_G \Vert K_{S^\star}(\cdot,x)^\star \Vert_{L^2(H\times_{\tau^\star \otimes \nu} Hom(W,E))} \, \Vert f(x) \Vert_W dx$$ is absolutely convergent. Hence, $\mathcal D$ contains the smooth vectors in $H^2(G,\tau)$, and   we have verified $b)$. Thus, we have shown Theorem~\ref{prop:symbrea}.  \end{proof}

  \begin{examp}\label{examp:sl2}  The following example shows that statement in Theorem~\ref{prop:symbrea} b) might not be sharp. Details for some of  the statements in this example are found in \cite{DaOZ} and references therein.  We set $G=SL_2(\mathbb R)$ and $H$  equal to the subgroup of  diagonal matrices in $G.$ We fix as compact Cartan subgroup $T$ the subgroup of orthogonal matrices in $G.$ We fix as positive root $\alpha$ one of the roots in $\Phi(\g,\t)$ and $\rho :=\frac12 \alpha.$ Then, the set of Harish-Chandra parameters for $G$ is $\{ n\rho, n \in \mathbb Z \backslash \{ 0\} \}.$ The lowest $K-$type for $\pi_{p\rho}^G$ ($p\geq 1$) is $(\pi_{(p+1)\rho}^K, \C).$ After we identify $G/K$ with the upper half plane $\mathcal H_+$ and we trivialize the vector bundle $G \times_{\pi_{(p+1)\rho}^K} \mathbb C \rightarrow G/K$ we have that $H^2(G,\pi_{(p+1)\rho}^K)$ can be identified with the space $$ \{ f : \mathcal H_+ \rightarrow \C : f \text{ holomorphic \, and } \, \int_{\mathcal H_+} \vert f(x+iy) \vert^2 y^{p+1} \frac{dx dy }{y^2} < \infty \}. $$ The bundle $H \times_{\pi_{(p+1)\rho}^K} \mathbb C \rightarrow H/L$ is also trivial and $L^2(H \times_{ res_L(\tau)} W)$ is identified with $L^2(i\mathbb R_{>0}, t^p dt).$ The restriction map $r : \,H^2(G,\tau) \rightarrow L^2(H \times_{ res_L(\tau)} W)$ becomes $r(f)(iy)=f(iy), y \in \mathbb R_{>0}.$ After we write the polar decomposition $r^\star =VQ$ we have $$ V(g)(z) =\int_0^\infty e^{izt} g(it) t^p dt, \,\, z \in \mathcal H_+, \,\, g \in L^2(H \times_{ res_L(\tau)} W).$$ Owing to Theorem~\ref{prop:symbrea}, for an integrable discrete series, equivalently  $\vert p \vert \geq 2,$ the linear map $
V^\star$ restricted to the subspace of smooth vectors in \\ $H^2(G, \pi_{(p+1)\rho}^K)$ is equal to   the integral linear map $$ V^\star (f)(it)=\int_{\mathcal H_+} f(x+iy) e^{-i\bar z t} y^{p-1} dx dy. \eqno{(\ddag)}$$
For the non-integrable discrete series  $H^2(G, \pi_{2\rho}^K ) $ the Lebesgue integral on the right of $(\ddag)$ does exists for any $K-$finite vector. In fact, \begin{multline*} \int_0^\infty \int_{-\infty}^\infty \left\rvert \frac{z-i}{z+i}\right\rvert^n \frac{1}{\vert z+i \vert^2 } \, e^{-yt}\, dx  dy \\ \leq \int_0^\infty \int_{-\infty}^\infty \frac{e^{-yt}}{x^2+(y+1)^2} dx dy <+\infty \,\, for \, \, t > 0, n \geq 0.\end{multline*} Whence, the following integral converges absolutely,
\begin{multline*} \int_{\mathcal H_+} \biggl( \frac{z-i}{z+i}\biggr)^n \frac{1}{ (z+i)^2 }\, e^{-i\bar{z}t}\,  dx dy \\ =\int_0^\infty \int_{-\infty}^\infty \frac{e^{-i(x-iy)t}}{(x+iy+i)^2} dx dy =c_n e^{-t}(e^{2t} (2t)^{-1}\frac{d^n  }{dt^n} (e^{-2t}(2t)^{n+1}) .
\end{multline*}  The inner integral is computed by means of Cauchy integral formula applied to compute the first derivative of $z \mapsto e^{-i(z-iy)t} $ at the point $-(iy+i)$. The path of integration is the half circumference of diameter $[-c,c]$ with $c$ positive and very large  so that $-(iy+i)$ is an interior point of the path.   We now have that for each $n\geq 0$, the resulting function defined by the right hand side of $(\ddag)$  belongs to $L^2(i\mathbb R_{>0}, tdt)$. Next,  the set of functions $( \frac{z-i}{z+i})^n \frac{1}{ (z+i)^2 }, n\geq 0$ span the subspace of $K-$finite vectors.  Therefore, functional analysis implies the right hand of $ (\ddag)$   evaluated in each of this set of generators  of $(H^2(SL_2(\mathbb R), \pi_{2\rho}^K ))_{K-fin} $ is equal to $V^\star$ evaluated at each of these generators.  This concludes the verification that $V^\star$ restricted to the subspace of $K-$finite vectors is an integral map. \end{examp}

 \section{Other model to realize  discrete series representations}
 We refer to the article \cite{Hi} for this section. Let $d_\star m$ denote a fixed  $G-$invariant Radon measure on $G/K$, after we normalize Haar measure on $K$ so that $K$ has volume one, and normalize Haar measure on $G$,    we have the equality $$\int_G f(x) dx =\int_{G/K} \int_K f(xk)dk\, d_\star m(xK).$$
 Owing to the Iwasawa  decomposition for $G=ANK$ the principal bundle $K \rightarrow G \rightarrow G/K $ is equivalent to the trivial bundle. Therefore, for each representation $\tau : K \rightarrow U(W)$, the vector bundle $ G\times_\tau W \rightarrow G/K$ is  parallelizable. We fix a   section $\sigma$ for $G \rightarrow G/K$ so that $\sigma (eK)=1,$ we  obtain a cocycle $c : G \times G/K \rightarrow Gl(W), \text{by}\, c(g,x)= \tau(\sigma (g\cdot x)^{-1} g\sigma(x)) $. That is, $c$ is a continuous map which satisfies $$ c(gh,x)=c(g,h \cdot x) c(h,x), g, h \in G, x \in G/K; \,\,c(k,eK)=\tau(k). $$
 We now recall how the above datum gives rise a unitary representation of $G$ equivalent to $  L^2(G\times_\tau W).$  Let $o=eK$. We consider the inner product on $W-$valued functions  on $G/K$ defined by $$ <F,H>_c= \int_{G/K} ((c(g,o) c(g,o)^\star )^{-1} F(z), H(z))_W \, d_\star m, \,\, z=g\cdot o.$$ The group $G$ acts unitarily on the corresponding space   of  square integrable functions $L_{c,\tau}^2(G/K, W)$ by the formula $$ g\cdot f(x)= c(g^{-1},x)^{-1} f(g^{-1} x).$$  We claim the following map determines an unitary  equivalence $$ L^2(G\times_\tau W) \ni f \mapsto F \in L_{c,\tau}^2(G/K,W) \, \text{where} \, F(g\cdot o):=c(g,o) f(g).$$ The inverse map is $$ L_{c,\tau}^2(G/K, W) \ni F \mapsto f \in L^2(G,\tau) :  f(g)= c(g,o)^{-1} F(g.o).$$
 The unitary equivalence is due to that for $F,H \in  L_{c,\tau}^2(G/K,W)$ the following  equality holds
 \begin{multline*}  \hskip 0.7cm \int_G ( f(g), h(g))_W dg  \\ =\int_{G/K} ([c(g,o) c(g,o)^\star ]^{-1} F(z),H(z))_W d_\star m(z)  .
 \end{multline*}

 For a convenient function $\tilde K : G/K \times G/K \rightarrow End_\C (W),$ we  consider the integral operator  $$ L_{c,\tau}^2(G/K,W) \ni F \mapsto ( G/K \ni z \mapsto \int_{G/K} \tilde K(w,z) F(w) d_\star m(w) \in W)$$ and  we define for $g,h \in G,  \tilde k(g,h):= c(h ,o)^{-1} \tilde K(g\cdot o,h\cdot o) c(  g,o), $ then, we have the commutative diagram

 $$\begin{array}{ccc}
 f & \rightarrow  & F \\
 \downarrow & \vdots &  \downarrow  \\
 \int \tilde k(x,y) f(x) dx  & \rightarrow  & \int \tilde K(w,z) F(w) d_\star m(w)
 \end{array} $$

 The kernel $\tilde k$ defines an intertwining linear operator for the left action of $G$ on $L^2(G\times_\tau W)$ if and only if $\tilde k$ is an invariant kernel, that is, $\tilde k(ga,gb)=\tilde k(a,b), a,b, g \in G$ and $\tilde k( gk,hk_1)=\tau(k_1^{-1}) \tilde k(g,h) \tau(k), g,h \in G, k, k_1 \in K. $  To an invariant kernel $\tilde k,$ the corresponding kernel $\tilde K$ satisfies  $$ \tilde K(t \cdot z, t \cdot w)= c(t,w)\tilde K(z,w) c(t,z)^{-1},  z,w \in G/K, t \in G. $$
 Next we consider the kernels $K_1,K_2$ which defines respective  intertwining linear operators on $L_{c,\tau}^2(G/K,W).$ Then readily follows the  equality
 \begin{multline*}K_1 (g \cdot x,g \cdot y) K_2 (g \cdot x,g \cdot y)^{-1} \\ =c(g,x) K_1(x,y) K_2(x,y)^{-1} c(g,x)^{-1}, x,y \in G/K, g \in G.
 \end{multline*}
 In particular, whenever $W$ is unidimensional, the kernel $K_1$ is equal to $K_2$ times a $G-$invariant function.

 For  each  Iwasawa decomposition $G=ANK $    a section of the principal bundle $G \rightarrow G/K$ is  $\sigma (anK)=an.$ Therefore, from the preceding considerations we obtain a unitary equivalence between $ L^2(G\times_\tau W) $ and $L_{c,\tau}^2(G/K,W)$ by mean of the cocycle $c$ associated to $\sigma$ and $\tau.$ When $G/K$ is a Hermitian symmetric space,  there is  another section $\sigma_+$  by means of the subgroups $P_-,P_+ , K_\C$ of the complex Lie group $G_\C$.  We have that $ G \subset P_+ K_\C P_- $ and $G/K$ is realized as a bounded domain in  $\p_+.$ Every $g \in G$ is uniquely written as $g=y_+(g)\mu(g)y_-(g)$, with $y_\pm(g) \in P_\pm,\, \mu(g) \in K_\C.$   Next, we fix $( \tau,W) $   irreducible representations of $K$ so that it is the lowest $K-$type of a holomorphic discrete series $\,H^2(G,\tau)$. By means of $\sigma_+$ and the representations $\tau$ we define $c_{+,\tau}=\tau (\mu(g))$ and we have that  the maps  $f \mapsto F=c_{+,\tau}(g,o) f(g)$  carries $\,H^2(G,\tau)$ onto the subspace  of holomorphic functions in $L_{c_{+,\tau}}^2(G/K, W)$.  Based on this model for the holomorphic discrete series,
 many authors have contributed to the study of   branching problems and harmonic analysis. Other authors have chosen different sections of  the principal bundle $G\rightarrow G/K$. Their choice, allowed them to analyze other discrete series,  branching problems, harmonic analysis. It is out of our knowledge to explicit all the work is done on the subject. We would like to call the attention of work, on holomorphic discrete series, of Jacobsen-Vergne, T. Kobayashi with his collaborators and the work of G.  Zhang on quaternionic discrete series \cite{LZ}.

 We would like to point out that after we fix respective smooth section for the principal bundle $G\rightarrow G/K$, (resp.  $H \rightarrow H/L$), we may translate the reproducing kernel $K_\lambda$ (resp. $K_\mu$) to a reproducing kernel $\tilde K_\lambda$  (resp. $\tilde K_\mu$) in $L_{c,\tau}^2 (G/K,W)$ (resp. $L_{c,\sigma}^2 (H/L, Z)$), the eigenspaces  of the Casimir operator to eigenspaces of the Casimir operator, and as we already indicated   kernel maps goes to kernel maps, it is a simple matter to verify differential operators correspond to differential operators and so on. We may say that each  statement  in this paper  has a correlative  in the language of the   spaces $L_{c,\tau}^2 $. For example, a function that corresponds to an element of $\,H^2(G,\tau) $ growth at most as the function $\Vert \tau (c(g,o)) \Vert. $ Differential operators intertwining $\,H^2(G,\tau)$ with $H^2(H,\sigma)$ corresponds to differential operators between the corresponding spaces.

\subsection{An example of discrete decomposition} \label{subsec:un1} We consider the groups
$G=U(1,n), H=U(1,n-1)\times U(1)$, $K=U(1)\times U(n),$ $ L=U(1)\times U(n-1) \times U(1)$ and we fix an integer $ \alpha $ positive and large. Then, $\tau_\alpha (k)=det (k)^\alpha, k \in K$ is a character of $K.$ To follow, we write the  decomposition of $\,H^2(G,\tau_\alpha)$ as an $H-$representation    and  we compute reproducing kernels, immersions, projections, etc.  Let $\mathcal D_n =\{ z \in \C^n, \vert z \vert <1\}$ denote the unit  ball in $\C^n.$   The group $G$ acts   means of fractional transformations, the action is transitive on   $\mathcal D_n$ and the isotropy subgroup at the origin of $\C^n$ is $K.$ We fix a Lebesgue measure $dm_n$ on $\C^n$ and define the measure  $d\mu_{\alpha} = (1-\vert z \vert^2)^{\alpha-(n+1)} dm_n. $
  Then, in \cite{DOZ} we find an explicit  isomorphism between the Hilbert spaces $H^2(G,\tau_\alpha) $ and   $V_{n,\alpha}:=\mathcal O (\mathcal D_n) \cap L^2 (\mathcal D_n,  d\mu_\alpha) .$  The unitary representation  of $G$    on $V_{n,\alpha}$ is by means of the action: $$ \pi_\alpha(g) f(z) = \tau_\alpha( J(g^{-1},z))^{-1} f(g^{-1} z),\,\, g \in G, z \in \mathcal D_n, f \in V_{n,\alpha}.$$
 For $ f\in \mathcal O(\mathcal D_n)$ we write the   convergent power series in $\mathcal D_n,$ $$f=f_0(z_1,\cdots,z_{n-1})+ f_1(z_1,\cdots,z_{n-1}) z_n+ f_2(z_1,\cdots,z_{n-1}) z_n^2 + \cdots .$$
For an integer $ m\geq 0$, we consider the linear subspace $$\tilde{\mathcal H}_m =\{ f \in \mathcal O(\mathcal D_n): \frac{\partial^p f}{\partial z_n^p}\vert_{\mathcal D_{n-1}} =0, \text{for} \,\,0\leq p \leq m-1 \}.$$ Then, $\mathcal H_m := \tilde{\mathcal H}_m \cap V_{n,\alpha}$ is a closed subspace in $V_{n,\alpha}.$  We denote $V_m$ for the orthogonal complement of $\mathcal H_{m+1} $ in $\mathcal H_m.$ Thus, a typical element of $\mathcal H_m$ (resp. $V_m$) is $$ f= f_m(z_1,\cdots z_{n-1})z_{n }^m + \, \cdots, \,\, (\text{resp.}\,\, f_m(z_1,\cdots z_{n-1})z_{n}^m ).$$ It readily follows that the action of $H$ on the a polynomial (resp. holomorphic function) in $z_1,\cdots,z_{n-1}$ (resp. in $z_n$) is again a polynomial (holomorphic function) in the same variables.  Whence, the subspaces $V_m$ are invariant for the action of $H.$ Therefore, we have the orthogonal decomposition $V_{n,\alpha} =V_0 +V_1 + V_2 +V_3 + \cdots $ \\
  The orthogonal projector $P_{\alpha,m}$  onto $V_m$ is given by

\begin{multline*}
  P_{\alpha,m}(f=f_0 +\cdots +f_k z_n^k +\cdots)(z)= f_m(z_1,\cdots,z_{n-1})z_n^m \\ =z_n^m \frac{\partial^m f}{\partial z_n^m} (z_1,\cdots,z_{n-1},0)  = \int_{\mathcal D_n} K_{\alpha,m}(w,z) f(w) d\mu_\alpha (w).
\end{multline*}
Here, $K_{\alpha,m}$ is the reproducing kernel for the subspace $V_m.$ Hence, $K_\lambda =\sum_{m\geq 0} K_{\lambda,m}.$ Let's write $z^\prime, w^\prime$ for vectors in $\C^{n-1}.$ Then, up to a constant \begin{multline*} K_\lambda (w,z)= \frac{1}{(1-(\tr w^\prime)^\star z^\prime- \bar w_n z_n)^\alpha} \\=\sum_{m\geq 0} \binom{-\alpha}{m} \frac{1}{(1-(\tr w^\prime)^\star z^\prime)^{\alpha +m} } (\bar w_n z_n)^m \end{multline*} Hence, the $m-$th summand is equal to $K_{\alpha,m}.$ \\ The representation of $H$ on $V_m$ is equivalent to $H^2(H,\tau_{\alpha +m}).$  An equivariant map $T_m$ from $H^2(H,\tau_{\alpha +m}) \equiv V_{n-1, \alpha +m}$ onto $V_m$ is given by $$ g(z_1, \cdots,z_{n-1})
 \mapsto g(z_1, \cdots,z_{n-1}) z_n^m.$$ A expression for $T_m$ as integral map is $$  T_m (g)(z^\prime, z_n)=\int_{\mathcal D_{n-1}} \frac{1}{(1-(\tr w^\prime)^\star z^\prime))^{\alpha +m}}\, z_n^m g(w_1,\cdots,w_{n-1}) d\mu_{\alpha +m}$$ This is due to that $\frac{1}{(1-(\tr w^\prime)^\star z^\prime))^{\alpha +m}}$ is the reproducing kernel for $V_{n-1,\alpha +m}.$  Finally, an intertwining map $S : V_{n,\alpha} \rightarrow H^2(H,\tau_{\alpha +m})$ is $$ S(f)(z_1,\dots,z_{n-1})=\frac{\partial^m f}{\partial z_n^m} (z_1,\dots,z_{n-1},0).$$

\section{Appendix}\label{sec:Apendix}
\subsection{Kernel linear maps}\label{subsec:A1} In this subsection we recall definitions and  facts on linear maps defined by kernels.
\begin{dfn} For   two measure spaces $(Y,\mu), (X,\nu)$, a linear transformation $T$ from $L^2(Y)$ into $L^2(X)$ is called an  {\it integral map, kernel map or an integral operator},  if there exists a function $K_T : Y\times X \rightarrow \C$ so that, for almost everywhere in  $x \in X$,
the function $y \mapsto K_T(y,x) g(y)$ belongs to $L^1(Y)$ for each function $g$ in the domain of $T$, and the equality   $Tg(x)=\int_Y K_T(y,x)g(y) dy$ holds  almost everywhere in $x \in X.$ \end{dfn} In an obvious way, the definition generalizes to linear maps between vector bundles. For a measurable kernel $k : Y\times X \rightarrow \C$, the domain $\mathcal D_{k} $ of the linear map defined by $k$ is is the set of $ f \in L^2(Y)$ so that $ \int_Y k(y,x)f(y) d\mu(y)$ converges for almost every $x \in X$ and the resulting function is square integrable with respect to $\nu.$  A kernel $k$ is {\it  a Carleman kernel}  if the function $y \mapsto k(y,x)$ is square integrable for almost every $x \in X$. An integral is a {\it Carleman map} if it is realized via a Carleman kernel. For a Carleman kernel $k$, the linear map defined on $\mathcal D_k$ is a closed linear map.  As usual,  $F^\star$ denotes the adjoint of a linear map $F.$ Formally, the adjoint of $T$ is an integral operator with kernel $K_{T^\star}(x,y):=K_T(y,x)^\star,$ however, even though when $T$ is  continuous, the linear map $T^\star$ might not be equal to an integral operator on the whole dual space, as example Appendix~\ref{subsec:A3} shows.
\subsection{Adjoint  to a  kernel linear map}\label{subsec:A2}   For an integral map $T$, if the adjoint linear map is an integral map of kernel $  K_{T^\star},$ then $ K_{T^\star} (x,y)=K_T(y,x)^\star.$
\subsection{Example}\label{subsec:A3}   For an example of an integral map whose adjoint is not an integral map, we consider  $X=\Z$ with discrete topology and usual Haar measure and $Y=S^1=\mathbb C/2\pi i\mathbb Z$ with Haar measure.  Point evaluation from $\ell^2(\Z)$ into $\C$ is continuous because $\vert a_n \vert \leq \Vert (a_k)_k \Vert_2. $ Thus, any continuous linear map from $L^2(S^1)$ into $\ell^2(\Z)$ is a Carleman kernel map. Next, we consider  $ T : L^2(S^1) \rightarrow \ell^2(\Z)$       given by the kernel $k_T(z,n)=\bar z^n$ (Fourier transform).  To follow, we verify $T^\star$ restricted to the subspace of absolutely convergent series is an integral map, and,   $T^\star$ is not an integral map.  It readily follows that the adjoint of $T$ is the linear map \\ \phantom{xxxxxxxxxxxx} $ \ell^2 (\Z) \ni (f_n)_n \mapsto \sum_n f_n z^n \in L^2(S^1).$  \\  We  notice  that the adjoint kernel $ k_{T}^\star (n,z)=z^n$ defines an integral map $ (T^\star)_0 : \ell^2(\Z)\rightarrow L^2(S^1)$  with domain the linear subspace  \\ \phantom{xxxxxxxxxxxxxx} $ \mathcal D_{k_T^\star} =\{(a_n)_n :\,\,\, \sum_n \vert a_n \vert < \infty \}.$ \\ In fact,  Lebesgue integral is absolutely convergent, whence $f$ belongs to $\mathcal D_{K_T^\star}$ iff  $\int_X \vert  k_{T }^\star(x,y) f(x)\vert dx <\infty \, \forall y  \in Y,$ and the resulting function belongs to $L^2(S^1)$. In our case, $\int_X \cdots =\int_\Z \vert f(n) K_T^\star (n,z)\vert  dn  =\sum_n \vert f_n \vert.$ Thus, the domain of $K_T^\star$ is as we have claimed.    Now, if $T^\star$ where an integral operator, Appendix~\ref{subsec:A2}, would give  the kernel that should represent  $T^\star$ ought to be $z^n$, hence,  $\mathcal D_{K_T^\star}$ would  be equal to $\ell^2(\Z).$ A contradiction.
\subsection{Reproducing kernel subspace}\label{subsec:A4} For the purpose of this paper, a closed subspace $V$ of $L^2(G\times_\tau W) $ is a {\it reproducing kernel subspace}, if $V$ consists of smooth functions and evaluation at each $x \in G$ is a continuous linear map.  Hence, for each $w \in W, x \in G $ the linear functional on $V$ $$ f \mapsto (e_x (f), w)$$ is represented by a function $k_x (\cdot)^\star(w) \in V.$ Thus, the function $y \mapsto k_x (y)^\star(w)$ is square integrable,   smooth and the following equality    hold: $$ (f(x), w)_W = (f, k_x(\cdot)^\star(w))_V, \,\, x \in G, \, w \in W, \, f \in V. \,\,\, $$ The function $$ W \ni w \mapsto k_x(y)^\star(w) \in W$$ is linear.    We define $K_V : G \times G \rightarrow End_\mathbb C(W)$ to be $ K_V (y,x) (w) =k_x(y) (w).$ Since the product of two square integrable functions gives an integrable function, we have, for $ f \in L^2( G\times_\tau W) $ the integral below is absolutely convergent  $$ P_V (f)(x):=\int_G K_V (y,x) f(y) dy \,\,x \in G  \eqno{(A-4.1)}$$   In \cite{Hi}, \cite{OO1}, \cite{OO2} we find a proof that the map $f \mapsto  P_V (f)$  is the orthogonal projector    onto $V.$ Let $j_W : W \rightarrow W^\star$ the conjugate linear map determined by the inner product $(..., ...)_W$. For each    orthonormal basis $\{f_j, j=1,2,\dots \} $   for $V.$ Then, in \cite{At} we find a proof of  $$K_V (y, x) = \sum_{r\geq 1} j_{W} (f_r(y)) \otimes f_r (x)= \sum_{r\geq 1} j_{V}(f_r)(y) \otimes f_r (x) \eqno{(A-4.2)}$$
We analyze the   convergence of the series (A-4.2) in Appendix~\ref{subsec:A5}.  \\ The main examples of reproducing kernel subspaces come from the statements below, in \cite{At} we find proofs of the stated facts.  \\ It readily follows that a continuous linear map $ T$ from $  L^2(H\times_\nu E) $ into a reproducing kernel space $V$ is a Carleman map.
\subsection{$L^2$-kernel of an elliptic operator}\label{subsec:A5} Let $D$ be   elliptic operator which maps sections of  vector bundle into sections of perhaps  another vector bundle over $G/K.$ Define $Ker_2(D)$ equal to the totality of $L^2-$sections which are in the kernel of $D$ as distributions. Then, $Ker_2(D)$ is closed in $L^2$ and owing to the regularity Theorem $Ker_2(D)$ consists of smooth functions, moreover, $L^2$ convergence of a sequence, implies uniform convergence on compact sets of the sequence as well as any derivative of the sequence, hence, $Ker_2(D)$ is an example of reproducing kernel Hilbert space.   In \cite[Prop 2.4]{At}, we find a proof of:  the matrix kernel $K_{Ker_2(D)} : G \times G \rightarrow End_\mathbb C (W)$ determined by $Ker_2(D)$  is a smooth function and  the convergence of the sequence in (A-4.2), as well as any derivative,  is uniform on compact sets.   In \cite{At}, we find a proof that $tr(K_{Ker_2(D)}(x,x))=\sum_j \Vert f_j(x) \Vert^2, x \in G/K$.  Whence,  Schwarz inequality yields the convergence of (A-4.2) is absolute.  As a Corollary,  we obtain. Let   $N$   be  $L^2-$closed subspace of $Ker_2(D)$. Then, $N$ is a reproducing kernel subspace and the matrix kernel for  $N$ is   a smooth function. This is so, because the series that represents the matrix kernel for $N$ is a sub-series of the absolutely convergent series  (A-4.2). For the same reason as before, the matrix kernel for $N$ is a real analytic function.  In particular, the matrix kernel  $K_\lambda $ for $\,H^2(G,\tau)$ is a real analytic function  and the matrix kernel for any closed  subspace in $\,H^2(G,\tau)$ is  also given by a real analytic function.
\subsection{Kernel for $P_\lambda$}\label{subsec:A6} The linear operator $P_\lambda$  has two representations as an integral operator. One {\it matrix  kernel} that represents $P_\lambda$ is $ K_\lambda (y,x)=   \Phi_0(x^{-1}y)$, where, $ z\mapsto \Phi_0(z):= d(\pi_\lambda) P_W \pi_\lambda^G(z)P_W $ is the spherical function associated to the lowest $K-$type $W.$ Let $\phi_0(z)= d(\pi_\lambda) tr(\Phi_0 (z)),$ then, the {\it the trace kernel} which represents $P_\lambda$  is \begin{multline*} k_\lambda (y,x) :=\phi_0(x^{-1}y):= d(\pi_\lambda)  tr P_W(\pi_\lambda^G(x^{-1}y)P_W)\\ = d(\pi_\lambda) \sum_j   (\pi_\lambda^G(x^{-1}y)f_j, f_j)_V =tr(K_\lambda(y,x). \end{multline*} where $f_1, \cdots$  is an orthonormal basis for the lowest $K-$type $V_\lambda[W].$ Of course, there are some identifications we have avoided to explicit. For a proof of the explicit representation of  $P_\lambda$ by $K_\lambda, k_\lambda$ we refer to \cite{OO2}, \cite{WW}. A sketch of proof is presented below.  Certainly, we recover $\Phi_0$ from $\phi_0$ by the formula $\Phi_0 (z)=\int_K \tau(k) \phi_0(k^{-1} z) dk. $\\
To follow, we sketch of proof that the kernel's $K_\lambda, k_\lambda$ represent $P_\lambda$ as an integral operator.
The orthogonal projector onto $\,H^2(G,\tau)$  is represented by the matrix kernel $K_\lambda (y,x)= d(\pi_\lambda)  P_W \pi_\lambda (x^{-1}y) P_W.$  In fact,  let $f_v$ be as in Proposition~\ref{prop:klmchar}, then  \begin{multline*}\int_G (P_W \pi_\lambda (x^{-1}y) P_W) f_v(y) dy  =\sum_{i,s} \int_G  (x^{-1}y w_i, w_s)(y^{-1}v,w_i) w_s dy \\ = \sum_{i,s} \int_G  (yw_i, xw_s) \overline{(yw_i, v)} w_s dy
= \frac{1}{d(\pi_\lambda)}\sum_{i,s}(w_i,w_i) \overline{(xw_s,v)} w_s\\ = \frac{1}{d(\pi_\lambda)}\sum_s(x^{-1}v,w_s) w_s = \frac{f_v(x)}{d(\pi_\lambda)}  . \end{multline*}
A trace kernel that represents
the orthogonal projector onto  $H^2(G,\tau)$ \\ is $k_\lambda (y,x) =  d(\pi_\lambda) \phi_0 (x^{-1}y) $.  Indeed, \begin{multline*} \int_G tr(P_W \pi_\lambda (x^{-1}y) P_W ) f_v(y) dy =  \int_G \sum_{r,i} (x^{-1}y w_r,w_r) (y^{-1}v,w_i) w_i dy  \\ =\sum_{r,i}\int_G (yw_r,xw_r) \overline{(y w_i,v)} dy =\frac{1}{d(\pi_\lambda)}\sum_{r,i}  (w_r,w_i) \overline{(xw_r,v)} w_i   =\frac{f_v(x)}{d(\pi_\lambda)}  .\end{multline*}

\section{Notation}\label{sec:notation}

\noindent
- $(\tau ,W),$ $(\sigma, Z)$, $L^2(G \times_\tau W), L^2(H\times_\sigma Z)$ (cf. Section~\ref{sec:prelim}).\\
-$\,H^2(G,\tau)=V_\lambda=V_\lambda^G $, $\pi_\lambda=\pi_\lambda^G$, $H^2(H,\sigma)=V_\mu^H, \pi_\mu^H.  $ (cf. Section~\ref{sec:prelim}).\\
- $d_\lambda =d(\pi_\lambda)$ dimension of $\pi_\lambda,$  $P_\lambda, P_\mu,  K_\lambda, k_\lambda, K_\mu, k_\mu, $  (cf. Section~\ref{sec:prelim}). \\ -$\Phi_0,\phi_0$ (cf. Appendix~\ref{subsec:A6}). \\
-$M_{K-fin} (resp.  \, M^\infty) $ $K-$finite vectors in $M$ (resp. smooth vectors in $M$).\\
-$dg,dh$ Haar measures on $G$, $H$.\\
-A unitary representation is square integrable, equivalently a discrete series representation,  (resp. integrable) if some nonzero  matrix coefficient is square integrable (resp. integrable) with respect to Haar measure. \\
-$\Theta_{\pi_\mu^H}(...)$  Harish-Chandra character of the representation $\pi_\mu^H.$\\
-For a module $M$ (resp. a simple submodule $N$)over a ring, $M[N]$ denotes the {\it isotypic component} of $N$ in $M$. That is, $M[N]$ is the sum of all irreducible submodules isomorphic to $N.$ If topology is involved, we define $M[N]$ to be the closure of $M[N].$ \\
-$  M_{H-disc}$ is the closure of the linear subspace spanned by the totality of $H-$irreducible submodules. $ M_{disc}:= M_{G-disc}$\\
-A representation $M$ is $H-$discretely decomposable if $ M_{H-disc} =M.$\\
-A representation is $H-$admissible if it is $H-$discretely decomposable and each isotypic component is equal to a finite sum of $H-$irreducible representations.\\
-$K_{\lambda,\mu}$ matrix kernel for the orthogonal projector $P_{\lambda,\mu}$ onto $\,H^2(G,\tau)[V_\mu^H].$ \\
-$U(\g) $ (resp. $\z(U(\g)$) universal enveloping algebra of the Lie algebra $\g$(resp. center of universal enveloping).\\
-$\N :=\{1,2,\cdots \}.$\\

\textbf{Acknowledgements}  The authors gratefully thank   the referees for the constructive comments and  recommendations which definitely help to improve the readability and quality of the paper. \\ The  authors would like to thank T. Kobayashi for much insight  and inspiration on the problems considered here. Also, we thank Michel Duflo, Birgit Speh, Yosihiki Oshima and Jan Frahm for conversations on the subject. The second author   thanks  Aarhus University for generous support, its hospitality and  excellent   working conditions during the preparation of this paper.
\providecommand{\MR}{\relax\ifhmode\unskip\space\fi MR }
\providecommand{\MRhref}[2]{%
  \href{http://www.ams.org/mathscinet-getitem?mr=#1}{#2}
}
\providecommand{\href}[2]{#2}

\end{document}